\documentclass[11pt,reqno]{amsart}
\usepackage{amsaddr}
\usepackage[pdftex]{graphicx}
\usepackage{orcidlink}
\usepackage{pgfplots}
\usepackage{geometry}
\usepackage{stix}
\usepackage{subcaption}
\usepackage{caption}
\usepackage{amsmath,amsthm}
\usepackage{cite}
\usepackage{hyperref}
\hypersetup{
    colorlinks=true, 
    linktoc=all,     
    allcolors=blue,  
}
\usepackage{url}
\newtheorem{theorem}{Theorem}[section]
\newtheorem{lemma}{Lemma}[section]
\newtheorem{proposition}{Proposition}[section]
\newtheorem{corollary}{Corollary}[section]

\theoremstyle{definition}
\newtheorem{definition}{Definition}[section]

\newtheorem{remark}{Remark}[section]
\numberwithin{equation}{section}
\numberwithin{equation}{subsection}
\pgfplotsset{compat=1.18}

\begin{document}
\pagenumbering{arabic}

\title[Parable of the Parabola]{Parable of the Parabola}
\author{Vladimir Dragovi\'{c}$^{1,2}$ and Mohammad Hassan Murad$^{1,3}$\orcidlink{0000-0002-8293-5242}}
\address{$^{1}$Department of Mathematical Sciences,
\\The University of Texas at Dallas,\\
800 W Campbell Rd, Richardson, Texas 75080\\
$^{2}$Mathematical Institute SANU, Belgrade, Serbia\\
$^{3}$Department of Mathematics and Natural Sciences\\
BRAC University, Dhaka, Bangladesh}
\email{vladimir.dragovic@utdallas.edu; mohammadhassan.murad@utdallas.edu}

\begin{abstract}
We study triangles and quadrilaterals which are inscribed in a circle and circumscribed about a parabola. Although these are particular cases of the celebrated \textit{Poncelet's Theorem}, in this paper we {\it do not assume} the theorem, but prove it along the way. Similarly, our arguments \emph{are logically independent} from the Cayley condition---describing points of a finite order on an elliptic curve or any other use of the theory of elliptic curves and functions. Instead, we use purely planimetric methods, including the Joachimsthal notation, to fully describe such polygons and associated circles and parabolas.  We prove that a circle contains the focus of a parabola if and only if there is a triangle inscribed in the circle and circumscribed about the parabola. We prove that if the center of a circle coincides with  the focus of a parabola, then there exists a quadrilateral inscribed in the circle and circumscribed about the parabola. We further prove that the quadrilaterals obtained in such a way are antiparallelograms. 
If the center of a circle does not coincide with  the focus of a parabola, then a quadrilateral inscribed in the circle and circumscribed about the parabola exists if and only if the directrix of the parabola contains the point of intersection of the polar of the focus with respect to the circle with the line determined by the center and the focus. That point coincides with the intersection of the diagonals of any quadrilateral inscribed in the circle and circumscribed about the parabola. In particular, for a given circle and a confocal pencil of parabolas with the focus different from the center of the circle, there is a unique parabola for which there exists a quadrilateral  circumscribed about it and inscribed in the circle.
\end{abstract}



\maketitle{}

\section{Introduction}\label{sec.1}\noindent
Quadratic functions with parabolas as their graphs and quadratic equations form one of the most important lessons in the entire mathematics curriculum. The aim of this parable is to emphasize once again the beauty, the depth, and the breath of the interrelations among these fundamental notions and concepts.

In this paper, we study triangles and quadrilaterals which are inscribed in a circle and circumscribed about a parabola. Although these are particular cases of the celebrated \emph{Poncelet's Theorem}, in this paper we {\it do not assume} the theorem, but prove it along the way. Similarly, our arguments {\it are logically independent} from the ones derived from Cayley's condition describing points of a finite order on an elliptic curve or any other use of the theory of elliptic curves and functions. Instead, we use pure planimetric methods, including the \emph{Joachimsthal's notation}, yet another effective but  maybe not widely known classical tool for studying tangents, polars; quadratic functions and equations, and Vieta formulas, to fully describe such polygons and associated circles and parabolas.  

We prove that a circle contains the focus of a parabola if and only if there is a triangle inscribed in the circle and circumscribed about the parabola. We prove that if the center of a circle coincides with  the focus of a parabola, then there exists a quadrilateral inscribed in the circle and circumscribed about the parabola. We further prove that the quadrilaterals obtained in such a way are antiparallelograms, also known as \emph{Darboux butterflies}. 

If the center of a circle does not coincide with  the focus of a parabola, then a quadrilateral inscribed in the circle and circumscribed about the parabola exists if and only if the directrix of the parabola contains the point of intersection of the polar of the focus with respect to the circle with the line determined by the center and the focus. That point coincides with the intersection of the diagonals of \emph{any} quadrilateral inscribed in the circle and circumscribed about the parabola. In particular, for a given circle and a confocal pencil of parabolas with the focus different from the center of the circle, there is a unique parabola for which there exists a quadrilateral  circumscribed about it and inscribed in the circle.

For two smooth conics $\mathcal{C}_1$, $\mathcal{C}_2$, suppose there is an $n$-gon inscribed in $\mathcal{C}_1$ and circumscribed about $\mathcal{C}_2$. Then, Poncelet's Theorem  (see, for example, \cite{Flatto2009,Dragovic-Radnovic2011}), states that for any point of $\mathcal{C}_1$, there exists such an $n$-gon, inscribed in $\mathcal{C}_1$ and circumscribed about $\mathcal{C}_2$, having this point as one of its vertices. (See Figure \ref{fig.1} for illustrations of the Poncelet's Theorem.)

\begin{figure}
  \centering
    \definecolor{qqzzqq}{rgb}{0.,0.6,0.}
  \begin{subfigure}[b]{0.45\textwidth}
    \centering
\begin{tikzpicture}[scale=1.5]
\clip(-2.3,-1.5) rectangle (1.5,2);
\draw [line width=1.pt] (-0.7407042547800149,0.6718312339789838) circle (1.cm);
\draw [samples=100,rotate around={-90.:(-0.5,0.)},xshift=-0.5cm,yshift=0.cm,line width=1.pt,domain=-6.0:6.0)] plot (\x,{(\x)^2/2/1.0});
\draw [line width=1.pt,color=red] (-1.612524993991828,1.161656303447838)-- (-0.3277435767612607,1.5825800898170195);
\draw [line width=1.pt,color=red] (-0.3277435767612607,1.5825800898170195)-- (-0.5411399388067768,-0.30805349586944775);
\draw [line width=1.pt,color=red] (-0.5411399388067768,-0.30805349586944775)-- (-1.612524993991828,1.161656303447838);
\draw [line width=1.pt,color=blue] (-1.737406077543382,0.7529821828346902)-- (-0.07552758912456228,1.4185172479550687);
\draw [line width=1.pt,color=blue] (-0.07552758912456228,1.4185172479550687)-- (-0.6684748428918675,-0.3255568108953732);
\draw [line width=1.pt,color=blue] (-0.6684748428918675,-0.3255568108953732)-- (-1.737406077543382,0.7529821828346902);
\draw [line width=1.pt,color=qqzzqq] (-1.6760019396349004,0.31796933901840574)-- (0.07655684552091517,1.2480987881019912);
\draw [line width=1.pt,color=qqzzqq] (0.07655684552091517,1.2480987881019912)-- (-0.8819634154457989,-0.3181414169200183);
\draw [line width=1.pt,color=qqzzqq] (-0.8819634154457989,-0.3181414169200183)-- (-1.6760019396349004,0.31796933901840574);
\begin{scriptsize}
\draw [fill=black] (0.,0.) circle (1.0pt);
\draw[color=black] (0.05757382715843051,-0.09039467618520017) node {$F$};
\draw [fill=black] (-0.740704254780661,0.6718312339790593) circle (1.0pt);
\draw[color=black] (-0.6,0.7787808173602851) node {$E$};
\draw[color=black] (-1.272528064479359,1.674294962225331) node {$\mathcal{D}$};
\draw [fill=red] (-1.612524993991828,1.161656303447838) circle (1.0pt);
\draw[color=black] (-1.75,1.3088461751664335) node {$A_1$};
\draw[color=black] (1.2757364506880895,1.7730649046736815) node {$\mathcal{P}$};
\draw [fill=red] (-0.3277435767612607,1.5825800898170195) circle (1.0pt);
\draw[color=black] (-0.2683669829211269,1.710510607789726) node {$B_1$};
\draw [fill=red] (-0.5411399388067768,-0.30805349586944775) circle (1.0pt);
\draw[color=black] (-0.5,-0.45) node {$C_1$};
\draw [fill=blue] (-1.737406077543382,0.7529821828346902) circle (1.0pt);
\draw[color=black] (-1.9,0.8) node {$A_2$};
\draw [fill=blue] (-0.07552758912456228,1.4185172479550687) circle (1.0pt);
\draw[color=black] (-0.011565132555414998,1.5458940370424752) node {$B_2$};
\draw [fill=blue] (-0.6684748428918675,-0.3255568108953732) circle (1.0pt);
\draw[color=black] (-0.6832007412041999,-0.47) node {$C_2$};
\draw [fill=qqzzqq] (-1.6760019396349004,0.31796933901840574) circle (1.0pt);
\draw[color=black] (-1.85,0.3) node {$A_3$};
\draw [fill=qqzzqq] (0.07655684552091517,1.2480987881019912) circle (1.0pt);
\draw[color=black] (0.19,1.35) node {$B_3$};
\draw [fill=qqzzqq] (-0.8819634154457989,-0.3181414169200183) circle (1.0pt);
\draw[color=black] (-0.91,-0.45) node {$C_3$};
\end{scriptsize}
\end{tikzpicture}
\caption{$n=3$.}
    \label{fig.1(A)}
  \end{subfigure}
  \hfill
  \begin{subfigure}[b]{0.45\textwidth}
    \centering
\begin{tikzpicture}[scale=1.5]
\clip(-3.5,-1.5) rectangle (1,3.5);
\draw [line width=1.pt] (-2.,2.) circle (1.cm);
\draw [samples=100,rotate around={-90.:(-0.875,0.)},xshift=-0.875cm,yshift=0.cm,line width=1.pt,domain=-7.0:7.0)] plot (\x,{(\x)^2/2/1.75});
\draw [line width=1.pt,color=red] (-2.885438743104727,2.4647560997008235)-- (-1.4966841922635603,2.8641025388705987);
\draw [line width=1.pt,color=red] (-1.4966841922635603,2.8641025388705987)-- (-1.198080482531653,1.4025679222678678);
\draw [line width=1.pt,color=red] (-1.198080482531653,1.4025679222678678)-- (-1.9197965821000609,1.0032214830980921);
\draw [line width=1.pt,color=red] (-1.9197965821000609,1.0032214830980921)-- (-2.885438743104727,2.4647560997008235);
\draw [line width=1.pt,color=blue] (-3.,2.)-- (-1.1912404131358199,2.5881393802962704);
\draw [line width=1.pt,color=blue] (-1.1912404131358199,2.5881393802962704)-- (-1.0769230769230766,1.6153846153846148);
\draw [line width=1.pt,color=blue] (-1.0769230769230766,1.6153846153846148)-- (-2.2318365099411035,1.0272452350883452);
\draw [line width=1.pt,color=blue] (-2.2318365099411035,1.0272452350883452)-- (-3.,2.);
\draw [line width=1.pt,color=qqzzqq] (-2.9106260468913074,1.5867685845399575)-- (-1.0411826064638312,2.284023248788734);
\draw [line width=1.pt,color=qqzzqq] (-1.0411826064638312,2.284023248788734)-- (-1.0107193888991184,1.8539730418728595);
\draw [line width=1.pt,color=qqzzqq] (-1.0107193888991184,1.8539730418728595)-- (-2.537471957745744,1.156718377624084);
\draw [line width=1.pt,color=qqzzqq] (-2.537471957745744,1.156718377624084)-- (-2.9106260468913074,1.5867685845399575);
\begin{scriptsize}
\draw [fill=black] (0.,0.) circle (1.0pt);
\draw[color=black] (0.06213041398675457,0.13856689093185776) node {$F$};
\draw [fill=black] (-2.,2.) circle (1.0pt);
\draw[color=black] (-1.938557215609118,2.139254520527726) node {$E$};
\draw[color=black] (-2.6783072635269196,2.9882858255242923) node {$\mathcal{D}$};
\draw[color=black] (0,2) node {$\mathcal{P}$};
\draw [fill=red] (-2.885438743104727,2.4647560997008235) circle (1.0pt);
\draw[color=black] (-3.05,2.55) node {$A_1$};
\draw [fill=red] (-1.4966841922635603,2.8641025388705987) circle (1.0pt);
\draw[color=black] (-1.35,3.) node {$B_1$};
\draw [fill=red] (-1.198080482531653,1.4025679222678678) circle (1.0pt);
\draw[color=black] (-1.09,1.3) node {$C_1$};
\draw [fill=red] (-1.9197965821000609,1.0032214830980921) circle (1.0pt);
\draw[color=black] (-1.9511665914258987,0.85) node {$D_1$};
\draw [fill=blue] (-3.,2.) circle (1.0pt);
\draw[color=black] (-3.1784791709258875,2.08) node {$A_2$};
\draw [fill=blue] (-1.1912404131358199,2.5881393802962704) circle (1.0pt);
\draw[color=black] (-1.05,2.65) node {$B_2$};
\draw [fill=blue] (-1.0769230769230766,1.6153846153846148) circle (1.0pt);
\draw[color=black] (-0.9003852733608395,1.55) node {$C_2$};
\draw [fill=blue] (-2.2318365099411035,1.0272452350883452) circle (1.0pt);
\draw[color=black] (-2.2706041121176765,0.8741138135773976) node {$D_2$};
\draw [fill=qqzzqq] (-2.9106260468913074,1.5867685845399575) circle (1.0pt);
\draw[color=black] (-3.05,1.55) node {$A_3$};
\draw [fill=qqzzqq] (-1.0411826064638312,2.284023248788734) circle (1.0pt);
\draw[color=black] (-0.9,2.3) node {$B_3$};
\draw [fill=qqzzqq] (-1.0107193888991184,1.8539730418728595) circle (1.0pt);
\draw[color=black] (-0.85,1.85) node {$C_3$};
\draw [fill=qqzzqq] (-2.537471957745744,1.156718377624084) circle (1.0pt);
\draw[color=black] (-2.6068541338984956,1.067457576101368) node {$D_3$};
\end{scriptsize}
\end{tikzpicture}
    \caption{$n=4$.}
    \label{fig.1(B)}
  \end{subfigure}
    \caption{Illustrations of Poncelet's Theorem.}
  \label{fig.1}
\end{figure}

\begin{definition}\label{def.1.1}
For a natural number $n \geq 3$, the ordered pair of conics $(\mathcal{C}_1,\mathcal{C}_2)$ is called an \emph{$n$-Poncelet pair} if there exists an $n$-gon inscribed in $\mathcal{C}_1$ and circumscribed about $\mathcal{C}_2$. An $n$-gon inscribed in $\mathcal{C}_1$ and circumscribed about $\mathcal{C}_2$ is called an \emph{$n$-Poncelet polygon}.
\end{definition}

Cyclic $n$-gons  circumscribed about central conics from a confocal pencil have been studied in a recent paper \cite{Dragovic-Radnovic24}. In \cite{Dragovic-Murad2025a}, parallel to this present paper, we studied $n$-Poncelet pairs in the case of circle and parabolas from a confocal family. Both of these papers were based on a heavy use of the Cayley condition. For more about the Cayley condition, see for example, \cite{Cayley1853a,Cayley1853b,Griffiths-Harris1978,Flatto2009,Dragovic-Radnovic2011} and references therein.

Geometric properties of Poncelet polygons have been extensively studied recently, in various instances, including those related to elliptical billiards, see, for example, \cite{Rezniketal.2021,Bialy-Tabachnikov2022,Helmanetal.2022,Helmanetal.2023,Garciaetal.2023, ReznikGarcia2023, Dragovic-Radnovic2025} and references therein.

In this paper, we restrict our discussion to the cases $n=3$ and $n=4$, that is, for triangles and quadrilaterals, where $\mathcal{C}_1$ is a circle and $\mathcal{C}_2$ is a parabola. Our study here uses a different approach, independent of the one we followed in \cite{Dragovic-Murad2025a}.  Although the two papers contain few common results on Poncelet pairs for $n=3$ and $n=4$, the corresponding proofs provided in these two papers are all completely different.

The present paper also provides several new results about the associated Poncelet triangles and quadrilaterals.  In this paper we consider \emph{real} Poncelet polygons only. On our way, we also device a general construction of the tangents from a given point to a parabola, given by its focus and directrix.

Let us recall {\bf the defining property of the parabola} as \emph{the set of points in the plane that have the same distance from a given point---the focus of the parabola, and a given line---the directrix of the parabola.}
See Figure \ref{fig.2}.

Parabolas have a well-known focal property:
\begin{lemma}[Focal property of parabola]\label{lemm.1.1}
Let $X$ be an arbitrary point on a parabola $\mathcal{P}$ with the focus $F$ and the directrix $\ell$. Then the tangent to the parabola $\mathcal{P}$ at $X$ is the bisector of the angle between $FX$ and the perpendicular from $X$ to $\ell$. See Figure \ref{fig.2}.
\end{lemma}

\begin{figure}
    \centering
\definecolor{wqwqwq}{rgb}{0.3764705882352941,0.3764705882352941,0.3764705882352941}
\begin{tikzpicture}[scale=0.8]
\clip(-5,-3) rectangle (5,7);
\draw [shift={(-0.6582909745456134,4.12022037241884)},line width=1.pt,color=wqwqwq,fill=wqwqwq,fill opacity=0.10000000149011612] (0,0) -- (-123.69006752597979:1.062886604697055) arc (-123.69006752597979:-102.30629726042541:1.062886604697055) -- cycle;
\draw [shift={(-0.6582909745456134,4.12022037241884)},line width=1.pt,color=wqwqwq,fill=wqwqwq,fill opacity=0.10000000149011612] (0,0) -- (-102.30629726042542:1.062886604697055) arc (-102.30629726042542:-80.92252699487108:1.062886604697055) -- cycle;
\draw [line width=1.pt,domain=-6.837445937674865:6.448636621038324] plot(\x,{(-4.-2.*\x)/3.});
\draw [samples=100,rotate around={-33.690067525979785:(-0.30769230769230765,-0.4615384615384615)},xshift=-0.30769230769230765cm,yshift=-0.4615384615384615cm,line width=1.pt,domain=-6.65640235470275:6.65640235470275)] plot (\x,{(\x)^2/2/1.1094003924504583});
\draw [line width=1.pt,domain=-6.837445937674865:6.448636621038324] plot(\x,{(--10.21531366847452--3.*\x)/2.});
\draw [line width=1.pt] (-0.6582909745456134,4.12022037241884)-- (-2.972764692724889,0.6485097951499262);
\draw [line width=1.pt] (0.,0.)-- (-0.6582909745456134,4.12022037241884);
\draw [line width=1.pt,domain=-6.837445937674865:6.448636621038324] plot(\x,{(--60.17631667266118--38.645941005423566*\x)/8.43062733694904});
\draw [shift={(-0.6582909745456134,4.12022037241884)},line width=1.pt,color=wqwqwq] (-123.69006752597979:1.062886604697055) arc (-123.69006752597979:-102.30629726042541:1.062886604697055);
\draw[line width=1.pt,color=wqwqwq] (-1.0513161131318058,3.194229497686935) -- (-1.0958095250472237,3.0894003420569085);
\draw [shift={(-0.6582909745456134,4.12022037241884)},line width=1.pt,color=wqwqwq] (-102.30629726042542:1.062886604697055) arc (-102.30629726042542:-80.92252699487108:1.062886604697055);
\draw[line width=1.pt,color=wqwqwq] (-0.6866315775862897,3.1146734218159633) -- (-0.6898399477418382,3.000837917974128);
\begin{scriptsize}
\draw [fill=black] (0.,0.) circle (1.5pt);
\draw[color=black] (-4,1) node {$\ell$};
\draw[color=black] (0.11,0.28794056111921335) node {$F$};
\draw[color=black] (2,0) node {$\mathcal{P}$};
\draw [fill=black] (-0.6582909745456134,4.12022037241884) circle (1.5pt);
\draw[color=black] (-0.4,4.05) node {$X$};
\draw[color=black] (1.1949399749643088,6.528603340126202) node {$f$};
\draw [fill=black] (-2.972764692724889,0.6485097951499262) circle (1.5pt);
\draw[color=black] (-3.25,0.55) node {$A'$};
\draw[color=black] (-0.09570804502497235,7.561121756117626) node {$i$};
\draw [fill=black] (-1.750602757005904,-0.8869297502630228) circle (1.5pt);
\draw[color=black] (-1.5533811028952194,-0.7142096661665804) node {$A$};
\end{scriptsize}
\end{tikzpicture}
    \caption{The defining and focal properties of the parabola with the focus $F$ and the directrix $\ell$: $|XA'|=|XF|$ and $\angle A'XA=\angle AXF$.}
    \label{fig.2}
\end{figure}

We will use the notation $\mathcal{D}(E,R)$ for the circle centered at $E$ with radius $R$, and $\mathcal{D}(E):=\mathcal{D}(E, 1)$ and $\mathcal{P}=\mathcal{P}(p)$ for a parabola from the confocal pencil  $\mathcal{F}$ of  parabolas with the focus at $F=(0,0)$ and $x$-axis as the axis of symmetry. Thus, we have the following  equations:
\begin{eqnarray}
\mathcal{D}(E,R)&:&  (x-x_E)^2+(y-y_E)^2=R^2,~(x_E,y_E)\in \mathbb{R}^2, \label{eq.1.0.1}\\
\mathcal{P}(p)&:&  y^2=2px+p^2,~p \in \mathbb{R}^*=\mathbb{R}\setminus \{0\}. \label{eq.1.0.2}\\
\mathcal{F}&:&  \{\mathcal{P}(p) \mid ~p \in \mathbb{R}^*\}. \label{eq.1.0.3}
\end{eqnarray}
where $|p|$ is the distance between the focus and the directrix, $\ell: x=-p$, of the parabola.

In the following Section \ref{sec.2} we review the classical Joachimsthal's notation for a pair of tangents from a point to a nondegenerate conic, derive the Joachimsthal section formula and discuss Joachimsthal's notation for plane conics, and apply this to parabolas.

\section{Joachimsthal's Notation}\label{sec.2}

We provide a brief review of the classical Joachimsthal's notation for a pair of tangents from a point to a conic. For more details, see the original reference by Joachimsthal \cite{Joachimsthal1871}, the classics by Sommerville \cite{Sommerville1924}, Spain \cite{Spain1957} and Brannan et al. \cite{Brannanetal.2012} as more recent and modern source.

Let $S(x,y)=0$ be a conic, where
\begin{equation*}
    S(x,y)=ax^2+2bxy+cy^2+2dx+2ey+f.
\end{equation*}
Let $A=(x_A,y_A),B=(x_B,y_B)$ be two points in $\mathbb{R}^2$. Define
\begin{eqnarray*}
    S_A(x,y)&:=&ax_A x+b (xy_A+x_Ay)+cy_Ay+d (x+x_A)+e(y+y_A)+f,\\
    S_{AA}&:=&S(x_A,y_A),\\
    S_{AB}&:=&S_A(x_B,y_B)=S_B(x_A,y_A)=S_{BA}.
\end{eqnarray*}

If a line $AB$ intersects the conic $S(x,y)=0$ at $T_1$ and $T_2$, then denote by $k_1$ and $k_2$ the real numbers in which $T_1$ and $T_2$, respectively, divide the segment $\overline{AB}$ in the ratio $k_i=|\overline{AT_i}|/|\overline{T_iB}|$ for $i=1,2$. See Figure \ref{fig.3}. Then, according to the Joachimsthal ratio lemma, $k_1$ and $k_2$ are the solutions of the following equation:
\begin{equation*}
    S\left(\frac{x_A+kx_B}{1+k},\frac{y_A+ky_B}{1+k}\right)=0.
\end{equation*}
This reduces to an quadratic equation in $k$:
\begin{equation}\label{eq.2.0.1}
    S_{BB}k^2+2S_{AB}k+S_{AA}=0.
\end{equation}
The eq. \eqref{eq.2.0.1} is called \emph{Joachimsthal's Section Equation}.

A line $AB$ is a tangent to the conic $S(x,y)=0$ if and only if eq. \eqref{eq.2.0.1} has a double zero. Thus, the line $AB$ is a tangent to the conic $S(x,y)=0$ if and only if the discriminant of the quadratic in $k$, eq. \eqref{eq.2.0.1}, is equal to zero; that is,
\begin{equation}\label{eq.2.0.2}
    S_{AA}S_{BB}=S_{AB}^2.
\end{equation}

Keeping $A$ fixed and assuming $B$ arbitrary in the previous equation, we get \emph{the Joachimsthal equation of the pair of tangents to the conic $S(x, y)=0$ from the point $A$}:

\begin{equation}\label{eq.2.0.3}
    SS_{AA}=S_{A}^2.
\end{equation}

\begin{remark}\label{rem.2.1}
Note that if $A$ lies in the convex complement of the parabola, that is  $S_{AA}<0$, then no real tangents can be drawn from $A$ to $\mathcal{P}(p)$. So, we consider $A \in \mathbb{R}^2$ such that $S_{AA} \geq 0$.
\end{remark}

\begin{figure}
  \centering
  \begin{subfigure}[b]{0.45\textwidth}
    \centering
\begin{tikzpicture}[scale=0.6]
\clip(-2,-6) rectangle (7,4);
\draw [samples=100,rotate around={-90.:(-0.5,0.)},xshift=-0.5cm,yshift=0.cm,line width=1.pt,domain=-10.0:10.0)] plot (\x,{(\x)^2/2/1.0});
\draw [line width=1.pt,domain=-3.273422175406992:16.231077989583365] plot(\x,{(--7.-7.*\x)/7.});
\begin{scriptsize}
\draw[color=black] (3.0457340676024565,2.4071491958169906) node {$\mathcal{P}$};
\draw [fill=black] (-1.,2.) circle (2.5pt);
\draw[color=black] (-0.8589385959884562,2.3505597369243687) node {$A$};
\draw [fill=black] (0,1) circle (2.5pt);
\draw[color=black] (-0.45,1) node {$T_1$};
\draw [fill=black] (4,-3) circle (2.5pt);
\draw[color=black] (3.7,-3.3) node {$T_2$};
\draw [fill=black] (6.,-5.) circle (2.5pt);
\draw[color=black] (6.139291153732456,-4.647670012796537) node {$B$};
\draw[color=black] (-0.2,1.7) node {$k_1$};
\draw[color=black] (3.5,-2) node {$k_2$};
\end{scriptsize}
\end{tikzpicture}
\caption{$k_1 \neq k_2$.}
    \label{fig.3(A)}
  \end{subfigure}
  \begin{subfigure}[b]{0.45\textwidth}
    \centering
\begin{tikzpicture}[scale=0.8]
\clip(-3,-4) rectangle (7,4);
\draw [samples=100,rotate around={-90.:(-0.5,0.)},xshift=-0.5cm,yshift=0.cm,line width=1.pt,domain=-8.0:8.0)] plot (\x,{(\x)^2/2/1.0});
\draw [line width=1.pt,domain=-5.555863684076075:13.948636480914278] plot(\x,{(-2.-2.*\x)/2.});
\begin{scriptsize}
\draw[color=black] (3.045734067602456,2.4071491958169906) node {$\mathcal{P}$};
\draw [fill=black] (-2.,1.) circle (2.0pt);
\draw[color=black] (-1.8586857030914432,1.350812629821382) node {$A$};
\draw [fill=black] (2.,-3.) circle (2.0pt);
\draw[color=black] (2.2,-2.8) node {$B$};
\draw [fill=black] (0.,-0.99) circle (2.0pt);
\draw[color=black] (-0.8,-1.1) node {$T_1=T_2$};
\end{scriptsize}
\end{tikzpicture}
    \caption{$k_1=k_2$.}
    \label{fig.3(B)}
  \end{subfigure}
  \caption{Joachimsthal's section.}
  \label{fig.3}
\end{figure}

Joachimsthal is nowadays more famous as a namesake of the Joachimsthal's first integrals in mechanics and geometry. These integrals are objects of recent studies, see for example \cite{Arnold-Tabachnikov2021,Dragovic-Gajic2023} and references therein.

\subsection{Joachimsthal's Notation and Parabola}\label{sec.2.1}
We apply Joachimsthal's notation to describe  the pair of tangents to a parabola $\mathcal{P}(p)$ from a point $A$. The pair of tangents  is given by the equation \eqref{eq.2.0.3},
where, in this case,
\begin{eqnarray}\label{eq.2.1.1}
    S(x,y)&=&y^2-2px-p^2.
\end{eqnarray}

We will use $m_{AB}$ to denote the slope of the line $AB$.

\begin{proposition}\label{prop.2.1}
Given a parabola $\mathcal{P}(p)\in \mathcal{F}$, a unit circle $\mathcal{D}(E)$ centered at $E$, and a point $A$, denote by $B,C$ and $B',C'$ the intersection points with the circle of the pair of tangents from $A$ to the parabola (Figure \ref{fig.4}.)
\begin{itemize}
 \item[(a)] The pair of tangents from $A$ to $\mathcal{P}(p)$ is given by:
\begin{equation}\label{eq.2.1.2}
    (2x_A+p)(y-y_A)^2-2y_A (y-y_A)(x-x_A)+p(x-x_A)^2=0.
\end{equation}
    \item[(b)] If $2x_A+p=0$ and $y_A \neq 0$, then one of the tangents from $A$ is vertical and the slope of the non-vertical tangent is given by 
    \begin{equation*}
        m_{AB}=\frac{p}{2y_A}, \qquad y_A \neq 0.
    \end{equation*} 
    If $2x_A+p=0$ and $y_A =0$, then $A\in \mathcal{D}(E) \cap \mathcal{P}(p)$ and $x=-p/2$ is the only vertical tangent at $A=(-p/2,0)$. 
    
    If $2x_A+p \neq 0$, the slopes $m_{AC}$ and $m_{AC'}$ of the lines $AC$ and $AC'$ are given by: 
\begin{equation}\label{eq.2.1.3}
    m_{AC}=\frac{p}{y_A + \sqrt{S_{AA}}}, \qquad     m_{AC'}=\frac{p}{y_A - \sqrt{S_{AA}}}.
\end{equation}
\item[(c)] The $x$-coordinates of $B$ and $C$ are the solutions of the quadratic equation in $x$:
\begin{eqnarray}\label{eq.2.1.4}
    &&(m_{AC}^2 + 1) x^2 - 2(m_{AC}^2 x_A+ x_E - m_{AC} (y_A - y_E) ) x + (m_{AC} x_A - y_A + y_E)^2\nonumber\\
    &&+ x_E^2 - 1=0.
\end{eqnarray}
Similarly, the $x$-coordinates of $B'$ and $C'$ can be calculated by the same equation \eqref{eq.2.1.4} if $m_{AC'}$ substitutes for $m_{AC}$.

\item[(d)] If $A\in \mathcal{D}(E)$, then $A$, $B$, and $B'$ coincide, and the $x$-coordinates of $C$ and $C'$ can be calculated by:
\begin{subequations}\label{eq.2.1.5}
\begin{eqnarray}
        x_{C}&=&\frac{(m_{AC}^2-1) x_A - 2m_{AC} (y_A - y_E) + 2x_E}{m_{AC}^2 + 1};\label{eq.2.1.5a}\\
    x_{C'}&=&\frac{(m_{AC'}^2-1) x_A - 2m_{AC'} (y_A - y_E) + 2x_E}{m_{AC'}^2 + 1}\label{eq.2.1.5b}.
\end{eqnarray}
\end{subequations}

The slope of the line $CC'$ is given by
\begin{equation}\label{eq.2.1.6}
    m_{CC'}=-\frac{(x_A^2+y_A^2)-(x_Ax_E+y_Ay_E)}{x_Ay_E-y_Ax_E}.
\end{equation}
\item[(e)] If $A\in \mathcal{D}(E)$ but $A\notin \mathcal{P}(p)$, then a necessary and sufficient condition for the line $AC$ to be the common tangent to $\mathcal{D}(E)$ and $\mathcal{P}(p)$ at $A$ is given by either: 
\begin{subequations}\label{eq.2.1.7}
\begin{equation}\label{eq.2.1.7a}
    2x_A+p=0,
\end{equation}
or
\begin{equation}\label{eq.2.1.7b}
 f(A,E,p):= p+2(x_A-x_E)(x_Ax_E+y_A y_E-x_E^2-y_E^2+1)=0.
\end{equation}
\end{subequations}
\end{itemize}
\end{proposition}

\begin{figure}
    \centering
\begin{tikzpicture}[scale=1.2]
\clip(-4,-3) rectangle (6,5);
\draw [line width=1.pt] (-2.,1.) circle (1.cm);
\draw [samples=100,rotate around={-90.:(-0.8,0.)},xshift=-0.8cm,yshift=0.cm,line width=1.pt,domain=-9.600000000000001:9.600000000000001)] plot (\x,{(\x)^2/2/1.6});
\draw [line width=1.pt,domain=-6.068884139096933:7.447964705616708] plot(\x,{(-6.77418171342396-3.132759308226526*\x)/4.088174426275572});
\draw [line width=1.pt,domain=-6.068884139096933:7.447964705616708] plot(\x,{(--19.591481881679677--3.132759308226526*\x)/8.179551677824357});
\begin{scriptsize}
\draw [fill=black] (-2.,1.) circle (1.0pt);
\draw[color=black] (-1.892069957891235,1.2528684435200215) node {$E$};
\draw[color=black] (-2.339035313075608,2.1776243507980335) node {$\mathcal{D}$};
\draw[color=black] (2.5313457985886,3.025317265802878) node {$\mathcal{P}$};
\draw [fill=black] (0.,0.) circle (1.0pt);
\draw[color=black] (0.11156784121112931,0.25104954396884127) node {$F$};
\draw [fill=black] (-3.525805394082394,1.0447983643824685) circle (1.0pt);
\draw[color=black] (-3.5,1.2) node {$A$};
\draw [fill=black] (-2.9658025028491366,1.259278856620325) circle (1.0pt);
\draw[color=black] (-3.032602243534119,1.4) node {$B$};
\draw [fill=black] (-1.4544953893457169,1.8381078210796737) circle (1.0pt);
\draw[color=black] (-1.36,2.04) node {$C$};
\draw [fill=black] (-2.9008705277173497,0.5659121145548236) circle (1.0pt);
\draw[color=black] (-3.032602243534119,0.37) node {$B'$};
\draw [fill=black] (-2.1848606127124324,0.017235249987270107) circle (1.0pt);
\draw[color=black] (-2.2003219269839063,-0.18050321276089784) node {$C'$};
\end{scriptsize}
\end{tikzpicture}
    \caption{The pair of tangents from $A$ to $\mathcal{P}$.}
    \label{fig.4}
\end{figure}

\begin{proof}
Let $F=(0,0)$ be the origin of the coordinate system and $E=(x_E,y_E)$ be the center of the circle $\mathcal{D}$.

\begin{itemize}
\item[(a)] By substituting \eqref{eq.2.1.1} into eq. \eqref{eq.2.0.3}, we get \eqref{eq.2.1.2}. 
\item[(b)] Denoting
\begin{equation*}
    m=\frac{y-y_A}{x-x_A},
\end{equation*}
one gets from eq. \eqref{eq.2.1.2} the following quadratic equation in $m$:
\begin{equation}\label{eq.2.1.8}
    (2x_A+p)m^2 -2y_A m+p=0.
\end{equation}
The solutions of the eq. \eqref{eq.2.1.8} are $m_{AC}$ and $m_{AC'}$. Thus, we have
\begin{subequations}\label{eq.2.1.9}
\begin{eqnarray}
    m_{AC}+m_{AC'}&=&\frac{2y_A}{2x_A+p},\label{eq.2.1.9a}\\
    m_{AC}m_{AC'}&=&\frac{p}{2x_A+p}\label{eq.2.1.9b}.
\end{eqnarray}
\end{subequations}
Solving for $m_{AC}$ and $m_{AC'}$ one obtains the eq. \eqref{eq.2.1.3}.

\item[(c)] If the line
\begin{equation*}
    y=m_{AC}(x-x_A)+y_A
\end{equation*}
intersects the circle $\mathcal{D}(E)$ at $B$ and $C$, then one gets quadratic equation \eqref{eq.2.1.4} which gives the $x$-coordinates of $B$ and $C$.

The coordinates are real or complex conjugates if the discriminant is positive or negative, respectively, where the discriminant is given by
\begin{equation*}
    4\left((m_{AC}^2 + 1)-(m_{AC} (x_A -  x_E) - (y_A - y_E))^2\right).
\end{equation*}
And the line $AC$ is a common tangent to both $\mathcal{D}(E)$ and $\mathcal{P}(p)$ if the discriminant is equal to zero.
\item[(d)] Let $A\in \mathcal{D}(E)$, thus $A=B=B'$ and
\begin{equation}\label{eq.2.1.10}
    x_A+x_C=\frac{2(m_{AC}^2 x_A+ x_E - m_{AC} (y_A - y_E) )}{m_{AC}^2 + 1}
\end{equation}
gives eq. \eqref{eq.2.1.5a}. Similarly, the eq. \eqref{eq.2.1.5b} is obtained if $m_{AC'}$ is used in the eq. \eqref{eq.2.1.10}.

After making the substitution of \eqref{eq.2.1.5} into the following equation for the slope of the line $CC'$:
\begin{equation}\label{eq.2.1.11}
        m_{CC'}=\frac{m_{AC'}(x_{C'}-x_A)-m_{AC}(x_{C}-x_A)}{x_{C'}-x_{C}},
\end{equation}
one gets
\begin{equation}\label{eq.2.1.12}
        m_{CC'}=\frac{(x_A - x_E) (1-m_{AC}m_{AC'}) + (m_{AC} + m_{AC'}) (y_A - y_E)}{(y_A - y_E) (1-m_{AC} m_{AC'}) - (m_{AC} + m_{AC'}) (x_A - x_E)}.
\end{equation}

A further use of \eqref{eq.2.1.9} into the eq. \eqref{eq.2.1.12} gives the eq. \eqref{eq.2.1.6}.

\item[(e)] The line $AC$, if not vertical, is tangent to both $\mathcal{D}(E)$ and $\mathcal{P}(p)$ at $A\in \mathcal{D}(E)$ if and only if
\begin{eqnarray*}
\frac{p}{y_A+\sqrt{S_{AA}}}=-\frac{x_A-x_E}{y_A-y_E}.
\end{eqnarray*}
This gives the eq. \eqref{eq.2.1.7b}.
\end{itemize}
\end{proof}

\subsection{Common Tangents}\label{sec.2.2}
In this section we study the common tangents to a parabola and a circle in more details. With the abuse of notation we will use the same letter to denote a conic and its corresponding $3 \times 3$ symmetric matrix.
\begin{proposition}\label{prop.2.2}
For a given parabola $\mathcal{P}(p)$ and a unit circle $\mathcal{D}$ centered at $E$, the points of common tangents either lie on the line $2x+p=0$ or at $\mathcal{D}(E) \cap \mathcal{H}(E,p)$ where 
\begin{equation*}
    \mathcal{H}(E,p):=\langle \mathbf{x},\mathcal{H}\mathbf{x}\rangle=0,
\end{equation*}
is a conic defined by the following $3\times 3$ symmetric matrix:
\begin{eqnarray}\label{eq.2.2.1}
    \mathcal{H}&=&\begin{pmatrix}
        2x_E & y_E & -2 x_E^2 - y_E^2  + 1\\
        y_E  & 0 & -x_E y_E\\
        -2 x_E^2 - y_E^2  + 1  & -x_E y_E & p + 2x_E(x_E^2 + y_E^2 - 1)
    \end{pmatrix},\quad \mathbf{x}=\begin{pmatrix}
        x \\ y \\ 1
    \end{pmatrix}.
\end{eqnarray}

\begin{itemize}
    \item[(a)] If $y_E \neq 0$, then $\mathcal{H}(E,p)$ is a hyperbola.
    \item[(b)] If $y_E = 0$ and $x_E \neq 0$, then $\mathcal{H}(E,p)$ is a pair of straight lines---parallel to the directrix---given by:
    \begin{equation}\label{eq.2.2.2}
        x=\frac{2x_E^2-1 \pm \sqrt{1-2px_E}}{2x_E}.
    \end{equation}
    \item[(c)] If $(x_E,y_E)= (0,0)$, then $\mathcal{H}(E,p)$ is a double straight line, given by:
    \begin{equation}\label{eq.2.2.3}
        x=-\frac{p}{2}.
    \end{equation}
    \item[(d)] For a fixed center $E$ of the circle, the conics $\mathcal{H}(E,p)$ form a pencil when the parabola $\mathcal {P}(p)$ runs through the confocal pencil of parabolas $\mathcal F$.
\end{itemize}
\end{proposition}
\begin{proof} If $A=(x_A,y_A)\in \mathcal{D}(E)$ be a point at which the circle and $\mathcal{D}$ and the parabola $\mathcal{P}$ has a common tangent, then, from the eq. \eqref{eq.2.1.7}, it follows that the coordinates of $A$ satisfy $2x+p=0$, or the following second-degree equation:
\begin{equation}\label{eq.2.2.4}
2x_E x^2 + 2y_E x y -2(2 x_E^2 -y_E^2 -1) x - 2x_E y_E y + p + 2x_E^3 - 2x_E y_E^2 - 2x_E=0.
\end{equation}
The symmetric matrix $\mathcal{H}$ that defines the conic in eq. \eqref{eq.2.2.4} is given in the eq. \eqref{eq.2.2.1}. 

Since
\begin{eqnarray*}
        \det \mathcal{H}&=& -y_E^2p,\\
\det \mathcal{H}'&=& -y_E^2,
\end{eqnarray*}
where
\begin{eqnarray*}
    \mathcal{H}'=\begin{pmatrix}
        2x_E & y_E\\
        y_E  & 0
    \end{pmatrix},
\end{eqnarray*}
the conic $\mathcal{H}(E,p)=0$ is nondegenerate if and only if $y_E \neq 0$ and it is a hyperbola. 
See Figure \ref{fig.5}. 

Moreover, the common asymptotes of the pencil of $\mathcal{H}(E,p)=0$ may be obtained by setting $p=0$ in eq. \eqref{eq.2.2.4}, and $E$ lies on one of the asymptotes. This concludes the proof of the proposition.
\end{proof}

\begin{figure}
    \centering
\definecolor{uuuuuu}{rgb}{0.26666666666666666,0.26666666666666666,0.26666666666666666}
\definecolor{wqwqwq}{rgb}{0.3764705882352941,0.3764705882352941,0.3764705882352941}
\begin{tikzpicture}[scale=1.3]
\clip(-3,-2) rectangle (2.5,4);
\draw [samples=100,rotate around={-90.:(-0.25,0.)},xshift=-0.25cm,yshift=0.cm,line width=1.pt,domain=-4.0:4.0)] plot (\x,{(\x)^2/2/0.5});
\draw [line width=1.pt] (-1.5,1.5) circle (1.cm);
\draw [samples=100,domain=-0.99:0.99,rotate around={157.5:(-1.5,0.8333333333333334)},xshift=-1.5cm,yshift=0.8333333333333334cm,line width=1.pt,color=wqwqwq,dash pattern=on 3pt off 3pt] plot ({0.3715793151639342*(1+(\x)^2)/(1-(\x)^2)},{0.8970718221660766*2*(\x)/(1-(\x)^2)});
\draw [samples=100,domain=-0.99:0.99,rotate around={157.5:(-1.5,0.8333333333333334)},xshift=-1.5cm,yshift=0.8333333333333334cm,line width=1.pt,color=wqwqwq,dash pattern=on 3pt off 3pt] plot ({0.3715793151639342*(-1-(\x)^2)/(1-(\x)^2)},{0.8970718221660766*(-2)*(\x)/(1-(\x)^2)});
\draw [line width=1.0pt,domain=-4.639025660947785:6.242464992929825] plot(\x,{(--76.45544895486323--2.7517871125629374*\x)/28.878783633732866});
\draw [line width=1.0pt,domain=-4.639025660947785:6.242464992929825] plot(\x,{(-0.4429309541568385-1.7302921280671346*\x)/0.2677478750023544});
\draw [line width=1.0pt,domain=-4.639025660947785:6.242464992929825] plot(\x,{(-1.3930031143493349-1.4756447376475408*\x)/2.4586141358737135});
\draw [line width=1.0pt,domain=-4.639025660947785:6.242464992929825] plot(\x,{(--3.3804851136429126--1.1208164198184054*\x)/3.7281876887243905});
\begin{scriptsize}
\draw [fill=black] (-1.5,1.5) circle (1.0pt);
\draw[color=black] (-1.4295972377535464,1.670794354136665) node {$E$};
\draw [fill=black] (0.,0.) circle (1.0pt);
\draw[color=black] (0.06813602640376523,0.17306108997935396) node {$F$};
\draw[color=black] (0.8322856509738221,0.8862674062447402) node {$\mathcal{P}$};
\draw[color=black] (-2.4688407271688235,2.23117074548804) node {$\mathcal{D}$};
\draw[color=wqwqwq] (-2.45,-0.19373072981427325) node {$\mathcal{H}$};
\draw [fill=uuuuuu] (-1.5948578329994039,2.495490829449803) circle (1.0pt);
\draw [fill=uuuuuu] (-0.5117616476063913,1.6529214139954451) circle (1.0pt);
\draw [fill=uuuuuu] (-2.0146178752493253,0.6425803580078949) circle (1.0pt);
\draw [fill=uuuuuu] (-1.2120959774782096,0.5423407318801887) circle (1.0pt);
\end{scriptsize}
\end{tikzpicture}
\begin{tikzpicture}[scale=1.3]
\clip(-3.2,-2) rectangle (2,2);
\draw [samples=100,rotate around={-90.:(-0.25,0.)},xshift=-0.25cm,yshift=0.cm,line width=1.pt,domain=-4.0:4.0)] plot (\x,{(\x)^2/2/0.5});
\draw [line width=1.pt] (-2.,0.) circle (1.cm);
\draw [line width=1.pt,color=wqwqwq,dash pattern=on 3pt off 3pt] (-2.183012701892219,-2.7154244908954603) -- (-2.183012701892219,6.28116375570934);
\draw [line width=1.pt,color=wqwqwq,dash pattern=on 3pt off 3pt] (-1.3169872981077808,-2.7154244908954603) -- (-1.3169872981077808,6.28116375570934);
\draw [line width=1.pt,domain=-4.639025660947785:6.242464992929825] plot(\x,{(--0.6779656687821338--1.2651011608411684*\x)/1.3528856829700262});
\draw [line width=1.pt,domain=-4.639025660947785:6.242464992929825] plot(\x,{(--12.709852971295646--1.7027974197884872*\x)/9.147114317029974});
\draw [line width=1.pt,domain=-4.639025660947785:6.242464992929825] plot(\x,{(-0.6779656687821338-1.2651011608411682*\x)/1.3528856829700262});
\draw [line width=1.pt,domain=-4.639025660947785:6.242464992929825] plot(\x,{(-12.709852971295648-1.702797419788487*\x)/9.147114317029974});
\draw [line width=1.pt,color=wqwqwq] (-3.2,0) -- (2,0);
\begin{scriptsize}
\draw [fill=black] (-2.,0.) circle (1.0pt);
\draw[color=black] (-1.9288416591393167,0.17306108997935396) node {$E$};
\draw [fill=black] (0.,0.) circle (1.0pt);
\draw[color=black] (0.06813602640376523,0.17306108997935396) node {$F$};
\draw[color=black] (0.8322856509738221,0.8862674062447402) node {$\mathcal{P}$};
\draw[color=black] (-3,0.6) node {$\mathcal{D}$};
\draw [fill=uuuuuu] (-1.3169872981077808,-0.7304064957637566) circle (1.0pt);
\draw [fill=uuuuuu] (-2.183012701892219,0.9831105486902831) circle (1.0pt);
\draw [fill=uuuuuu] (-1.3169872981077808,0.7304064957637565) circle (1.0pt);
\draw [fill=uuuuuu] (-2.183012701892219,-0.9831105486902832) circle (1.0pt);
\draw[color=wqwqwq] (-2.05,-1.9) node {$\mathcal{H}$};
\end{scriptsize}
\end{tikzpicture}
    \caption{Points of common tangents at the circle either lie on a hyperbola or pair of straight lines.}
    \label{fig.5}
\end{figure}

Since the conic $\mathcal{H}=0$ and the circle $\mathcal{D}$ can have at most four common points, so we have the following: 
\begin{proposition}\label{prop.2.3}
For a given parabola $\mathcal{P}(p)$ with focus $F=(0,0)$ and a given circle $\mathcal{D}$  centered at $E \in \mathbb{R}^2$, there exist at most four common tangents.  More precisely, if $A\in \mathcal{D}(E)$ is a point of a common tangent, then the $x$-coordinate of $A$ is a root of the following quartic polynomial:
\begin{eqnarray}\label{eq.2.2.5}
    ax^{4} +bx^3 + c x^2 + dx + e=0,
\end{eqnarray}
where
\begin{eqnarray*}
    a&=& 4(x_E^{2} + y_E^{2}),\\
    b&=& -8x_E (2x_E^{2} + 2y_E^{2} -1),\\
    c&=& 4 (6x_E^4 + 6x_E^2 y_E^2 - 6x_E^2 - y_E^2 + p x_E + 1),\\
    d&=& -4 (2x_E^2 - 1) (2x_E^3 + 2x_E y_E^2 + p - 2x_E),\\
    e&=& 4x_E^2 (x_E^4 + x_E^2 y_E^2 - 2x_E^2 - y_E^2 + p x_E)+(p - 2x_E)^2.
\end{eqnarray*}
\end{proposition}

\begin{proof} For a given parabola $\mathcal{P}(p)$, let $A\in \mathcal{D}(E)$ be a point of a common tangent of the circle and the parabola.

If $y_E \neq 0$, then by solving \eqref{eq.2.1.7b} for $y_A$, we obtain
\begin{equation*}
    y_A = \frac{-p - 2x_A^2 x_E + 4x_A x_E^2 + 2x_A y_E^2 - 2x_A - 2x_E^3 - 2x_E y_E^2 + 2x_E}{2y_E (x_A - x_E)}.
\end{equation*}
Now substituting $y_A$ into the following equation:
\begin{equation*}
    (x_A-x_E)^2+(y_A-y_E)^2=1,
\end{equation*}
we finally obtain the eq. \eqref{eq.2.2.5}.

Observe that if $x_A=x_E$, then the tangent  to the circle is parallel to the axis of the parabola, hence, cannot be a common tangent.

If $y_E=0$, then the eq. \eqref{eq.2.2.4} reduces to
\begin{equation*}
    2x_E x_A^2 -2 (2 x_E^2 -1) x_A + p + 2x_E(x_E^2 - 1)=0,
\end{equation*}
whose solutions are given by the eq. \eqref{eq.2.2.2}.
\end{proof}

\begin{corollary}\label{cor.2.1}
    If $E=F=(0,0)$, then the points of common tangents of the unit circle centered at $F$ and the parabola $\mathcal {P}(p)$, that belong to the circle, are:
    \begin{equation*}
        (x_A,y_A)=\left(-\frac{p}{2},\pm \frac{\sqrt{4-p^2}}{2}\right),\qquad p\in [-2,0)\cup (0,2].
    \end{equation*}
\end{corollary}

We will return to these considerations in Remark \ref{rem.4.5}.

\section{Triangles inscribed in a circle circumscribed about a parabola}\label{sec.3}
In this section we want to develop the sufficient and necessary conditions under which a triangle is inscribed in the circle and circumscribed about a parabola. We will use Joachimsthal's equation \eqref{eq.2.0.3} and the following proposition to derive such a condition. We further assume that the circle and the parabola are nowhere tangent to each other, that is, if $S_{AA}=0$ then $f(A,E,p)\neq 0$, or if $f(A,E,p)= 0$ then $S_{AA} \neq 0$. See Figure \ref{fig.6}.

\begin{figure}
    \centering
\begin{tikzpicture}[scale=1.3]
\clip(-3,-2) rectangle (2,3);
\draw [samples=100,rotate around={-90.:(-0.5,0.)},xshift=-0.5cm,yshift=0.cm,line width=1.pt,domain=-6.0:6.0)] plot (\x,{(\x)^2/2/1.0});
\draw [line width=1.pt,domain=-4.347663858223659:6.628613118222597] plot(\x,{(--0.5785216653585711--1.*\x)/0.39628692978336233});
\draw [line width=1.pt] (-1.359677347968845,0.7735873609408516) circle (1.cm);
\begin{scriptsize}
\draw [fill=black] (0.,0.) circle (1.0pt);
\draw[color=black] (0.07333659062274955,0.1688758052940449) node {$F$};
\draw[color=black] (1,1.5) node {$\mathcal{P}$};
\draw [fill=black] (-0.4214783346414117,0.3962869297834293) circle (1.0pt);
\draw[color=black] (-0.3,0.3) node {$A$};
\draw [fill=black] (-1.359677347968845,0.7735873609408516) circle (1.0pt);
\draw[color=black] (-1.2885348120103972,0.9412804814143358) node {$E$};
\draw[color=black] (-1.9999601715948767,1.744174815802533) node {$\mathcal{D}$};
\end{scriptsize}
\end{tikzpicture}
\begin{tikzpicture}[scale=1.3]
\clip(-1,-2) rectangle (2,2);
\draw [samples=100,rotate around={-90.:(-0.5,0.)},xshift=-0.5cm,yshift=0.cm,line width=1.pt,domain=-8.0:8.0)] plot (\x,{(\x)^2/2/1.0});
\draw [line width=1.pt] (-0.5,-4.97027492673885) -- (-0.5,4.114038108787601);
\draw [line width=1.pt] (0.5,0.) circle (1.cm);
\begin{scriptsize}
\draw [fill=black] (0.,0.) circle (1.0pt);
\draw[color=black] (0.07523583822330626,0.18503056079938954) node {$F$};
\draw[color=black] (1.3,1.78) node {$\mathcal{P}$};
\draw [fill=black] (0.5,0.) circle (1.0pt);
\draw[color=black] (0.5808798563049999,0.2080143798031029) node {$E$};
\draw [fill=black] (-0.5,0.) circle (1.0pt);
\draw[color=black] (-0.7,0) node {$A$};
\draw[color=black] (0.09821965722701961,0.7251503073866531) node {$\mathcal{D}$};
\end{scriptsize}
\end{tikzpicture}
\begin{tikzpicture}[scale=1.3]
\clip(-1,-2) rectangle (2,2);
\draw [samples=100,rotate around={-90.:(-0.25,0.)},xshift=-0.25cm,yshift=0.cm,line width=1.pt,domain=-5.0:5.0)] plot (\x,{(\x)^2/2/0.5});
\draw [line width=1.pt] (-0.25,-4.97027492673885) -- (-0.25,4.114038108787601);
\draw [line width=1.pt] (0.75,0.) circle (1.cm);
\begin{scriptsize}
\draw [fill=black] (0.,0.) circle (1.0pt);
\draw[color=black] (0.07109667466227554,0.16858618452441404) node {$F$};
\draw[color=black] (1.5,1.2) node {$\mathcal{P}$};
\draw [fill=black] (0.75,0.) circle (1.0pt);
\draw[color=black] (0.8221214216537042,0.1891622049899326) node {$E$};
\draw [fill=black] (-0.25,0.) circle (1.0pt);
\draw[color=black] (-0.4,0) node {$A$};
\draw[color=black] (0.30772091001573937,0.631546644998582) node {$\mathcal{D}$};
\end{scriptsize}
\end{tikzpicture}
\caption{The circle $\mathcal{D}$ and parabola $\mathcal{P}$ are tangential at $A$.}
\label{fig.6}
\end{figure}

\begin{remark}\label{rem.3.1}
For two smooth conics $\mathcal{C}_1,\,\mathcal{C}_2$, suppose $A\in \mathcal{C}_1$ and the pair of tangents from $A$ to $\mathcal{C}_2$ intersect $\mathcal{C}_1$ at $B,C$.  We say, the $\triangle ABC$ circumscribes $\mathcal{C}_2$ \textit{non-trivially} if $A,B,C$ are distinct and the lines $AB,BC,AC$ are tangent to $\mathcal{C}_2$; we also say the $\triangle ABC$ circumscribes $\mathcal{C}_2$ \textit{trivially} if $A\in \mathcal{C}_1\cap \mathcal{C}_2$ implies $B\,(=C)$ is a point of common tangent, or $A$ is a point of common tangent implies $B \in \mathcal{C}_1 \cap \mathcal{C}_2$. See Figures \ref{fig.7(B)} and \ref{fig.7(D)} for the non-trivial and trivial triangles, respectively.
\end{remark}

\begin{figure}
  \centering
  \begin{subfigure}[b]{0.45\textwidth}
    \centering
\begin{tikzpicture}[scale=1.4]
\clip(-3.5,-2) rectangle (2,3);
\draw [samples=100,rotate around={-90.:(-0.5,0.)},xshift=-0.5cm,yshift=0.cm,line width=1.pt,domain=-6.0:6.0)] plot (\x,{(\x)^2/2/1.0});
\draw [line width=1.pt] (-2.,1.) circle (1.cm);
\draw [line width=1.pt] (-2.9173458273638944,1.3980912370525787)-- (-1.3775710943421984,1.78267634268683);
\draw [line width=1.pt] (-1.3775710943421984,1.78267634268683)-- (-1.002638707015863,0.9274021263602662);
\draw [line width=1.pt] (-2.9173458273638944,1.3980912370525787)-- (-1.4379037846051317,0.17292815025611685);
\draw [line width=1.pt,dash pattern=on 3pt off 3pt,domain=-1.4379037846051317:6.908968826617299] plot(\x,{(-1.6967212099717583-1.3804751770174157*\x)/1.6669886955251583});
\draw [line width=1.pt,dash pattern=on 3pt off 3pt,domain=-1.002638707015863:6.908968826617299] plot(\x,{(-0.8141232307896593-1.365778941853062*\x)/0.5987258231966655});
\draw [line width=1.pt,dash pattern=on 3pt off 3pt,domain=-1.3775710943421984:6.908968826617299] plot(\x,{(--18.912100969748757--2.221053158179626*\x)/8.89249605239638});
\begin{scriptsize}
\draw[color=black] (1.5923099209147562,1.8802222436461216) node {$\mathcal{P}$};
\draw [fill=black] (-2.,1.) circle (1.0pt);
\draw[color=black] (-1.9228364630208095,1.1991626317586053) node {$E$};
\draw[color=black] (-2.40616909081195,2.143858222441289) node {$\mathcal{D}$};
\draw [fill=black] (-2.9173458273638944,1.3980912370525787) circle (1.0pt);
\draw[color=black] (-2.9663955457516806,1.64954076220035) node {$A$};
\draw [fill=black] (-1.3775710943421984,1.78267634268683) circle (1.0pt);
\draw[color=black] (-1.296701013382287,1.968100903244511) node {$B$};
\draw [fill=black] (-1.4379037846051317,0.17292815025611685) circle (1.0pt);
\draw[color=black] (-1.4724583325790652,0.001815894730552145) node {$C$};
\draw [fill=black] (-1.002638707015863,0.9274021263602662) circle (1.0pt);
\draw[color=black] (-0.75,1.) node {$C'$};
\draw [fill=black] (0.,0.) circle (1.0pt);
\draw[color=black] (0.07640304284254332,0.17757321392733055) node {$F$};
\end{scriptsize}
\end{tikzpicture}
\caption{$C \neq C'$.}
    \label{fig.7(A)}
  \end{subfigure}
  \hfill
  \begin{subfigure}[b]{0.45\textwidth}
    \centering
\begin{tikzpicture}[scale=1.5]
\clip(-2.5,-2) rectangle (2,3);
\draw [samples=100,rotate around={-90.:(-0.5,0.)},xshift=-0.5cm,yshift=0.cm,line width=1.pt,domain=-6.0:6.0)] plot (\x,{(\x)^2/2/1.0});
\draw [line width=1.pt] (-0.78,0.62) circle (1.cm);
\draw [line width=1.pt] (-1.7586632485394569,0.8254707910098101)-- (-0.18836319310593622,1.426204619639578);
\draw [line width=1.pt] (-0.18836319310593622,1.426204619639578)-- (-0.6157972699694686,-0.3664266133121715);
\draw [line width=1.pt] (-1.7586632485394569,0.8254707910098101)-- (-0.6115263025363638,-0.3657061495511383);
\draw [line width=1.pt,dash pattern=on 3pt off 3pt,domain=-0.6115263025363638:6.908968826617299] plot(\x,{(-0.5756461726816599-0.5973220177527756*\x)/0.5752379280467311});
\draw [line width=1.pt,dash pattern=on 3pt off 3pt,domain=-0.18836319310593622:6.908968826617299] plot(\x,{(--4.6517858325741335--1.187765129683956*\x)/3.104782118295206});
\draw [line width=1.pt,dash pattern=on 3pt off 3pt,domain=-4.251620942378122:-0.18836319310593622] plot(\x,{(-0.9410472157764016-1.7919107691907161*\x)/-0.42316310943042756});
\begin{scriptsize}
\draw[color=black] (1.5923099209147562,1.8802222436461216) node {$\mathcal{P}$};
\draw [fill=black] (-0.7804138882417508,0.6224589281441759) circle (1.0pt);
\draw[color=black] (-0.7035200610931603,0.825678328465451) node {$E$};
\draw[color=black] (-1.1868526888843007,1.7703739191481351) node {$\mathcal{D}$};
\draw [fill=black] (-1.7586632485394569,0.8254707910098101) circle (1.0pt);
\draw[color=black] (-1.8020033060730247,1.0673446423610213) node {$A$};
\draw [fill=black] (-0.18836319310593622,1.426204619639578) circle (1.0pt);
\draw[color=black] (-0.11033910880403361,1.6056014324011554) node {$B$};
\draw [fill=black] (-0.6115263025363638,-0.3657061495511383) circle (1.0pt);
\draw[color=black] (-0.8133683855911468,-0.4815167330605885) node {$C$};
\draw [fill=black] (-0.6157972699694686,-0.3664266133121715) circle (1.0pt);
\draw[color=black] (-0.5277627418963821,-0.6353043873577696) node {$C'$};
\draw [fill=black] (0.,0.) circle (1.0pt);
\draw[color=black] (0.14231203754133515,-0.10803242976743437) node {$F$};
\end{scriptsize}
\end{tikzpicture}
    \caption{$C=C'$.}
    \label{fig.7(B)}
  \end{subfigure}
\vspace{1em}
  \begin{subfigure}[b]{0.45\textwidth}
    \centering
\begin{tikzpicture}[scale=1.4]
\clip(-2,-2) rectangle (2,3);
\draw [samples=100,rotate around={-90.:(-0.5,0.)},xshift=-0.5cm,yshift=0.cm,line width=1.pt,domain=-8.0:8.0)] plot (\x,{(\x)^2/2/1.0});
\draw [samples=100,rotate around={-90.:(-0.5,0.)},xshift=-0.5cm,yshift=0.cm,line width=1.pt,domain=-8.0:8.0)] plot (\x,{(\x)^2/2/1.0});
\draw [line width=1.pt] (-0.7410486308546327,1.144622513551221) circle (1.cm);
\draw [line width=1.pt,domain=-4.761383944821422:14.240133299231067] plot(\x,{(--1.2553746049563115--1.*\x)/1.2291253841299605});
\draw [samples=100,domain=-0.99:0.99,rotate around={151.45986285693243:(-0.7410486308548296,0.27097204087612314)},xshift=-0.7410486308548296cm,yshift=0.27097204087612314cm,line width=1.pt,dash pattern=on 3pt off 3pt] plot ({0.6893085244575193*(1+(\x)^2)/(1-(\x)^2)},{1.2674302459313938*2*(\x)/(1-(\x)^2)});
\draw [samples=100,domain=-0.99:0.99,rotate around={151.45986285693243:(-0.7410486308548296,0.27097204087612314)},xshift=-0.7410486308548296cm,yshift=0.27097204087612314cm,line width=1.pt,dash pattern=on 3pt off 3pt] plot ({0.6893085244575193*(-1-(\x)^2)/(1-(\x)^2)},{1.2674302459313938*(-2)*(\x)/(1-(\x)^2)});
\draw [samples=100,domain=-0.99:0.99,rotate around={151.45986285693243:(-0.7410486308548296,0.27097204087612314)},xshift=-0.7410486308548296cm,yshift=0.27097204087612314cm,line width=1.pt,dash pattern=on 3pt off 3pt] plot ({0.6893085244575193*(1+(\x)^2)/(1-(\x)^2)},{1.2674302459313938*2*(\x)/(1-(\x)^2)});
\draw [samples=100,domain=-0.99:0.99,rotate around={151.45986285693243:(-0.7410486308548296,0.27097204087612314)},xshift=-0.7410486308548296cm,yshift=0.27097204087612314cm,line width=1.pt,dash pattern=on 3pt off 3pt] plot ({0.6893085244575193*(-1-(\x)^2)/(1-(\x)^2)},{1.2674302459313938*(-2)*(\x)/(1-(\x)^2)});
\begin{scriptsize}
\draw[color=black] (1.4049411651141246,1.8) node {$\mathcal{P}$};
\draw [fill=black] (0.,0.) circle (1.0pt);
\draw[color=black] (0.12821305865169844,0.14632979200879226) node {$F$};
\draw [fill=black] (-0.7410486308546327,1.144622513551221) circle (1.0pt);
\draw[color=black] (-0.6459731335648791,1.395893470674145) node {$E$};
\draw[color=black] (-1.2028439034048735,2.2) node {$\mathcal{D}$};
\draw [fill=black] (0.2553746049563115,1.2291253841299605) circle (1.0pt);
\draw[color=black] (0.3455284810282816,1.4502223262682907) node {$A$};
\draw[color=black] (-1.6,0) node {$\mathcal{H}$};
\draw [fill=black] (-1.4455478022271646,0.43491770341152913) circle (1.0pt);
\draw[color=black] (-1.678221389853649,0.4451384977765941) node {$C'$};
\draw [fill=black] (-1.0264820083182127,0.18622396021877063) circle (1.0pt);
\draw[color=black] (-0.95,-0.0302389886721815) node {$B=C$};
\end{scriptsize}
\end{tikzpicture}
\caption{$C \neq C'$.}
    \label{fig.7(C)}
  \end{subfigure}
  \hfill
  \begin{subfigure}[b]{0.45\textwidth}
    \centering
\begin{tikzpicture}[scale=1.5]
\clip(-2,-2) rectangle (2,2.5);
\draw [samples=100,rotate around={-90.:(-0.5,0.)},xshift=-0.5cm,yshift=0.cm,line width=1.pt,domain=-8.0:8.0)] plot (\x,{(\x)^2/2/1.0});
\draw [line width=1.pt] (-0.6595553474634155,0.750738310493664) circle (1.cm);
\draw [line width=1.pt,domain=-4.747801730922887:14.253715513129604] plot(\x,{(--1.2297335976205477--1.*\x)/1.208084101062958});
\draw [samples=100,domain=-0.99:0.99,rotate around={155.6503365434604:(-0.6595553474635177,-0.5812837616745371)},xshift=-0.6595553474635177cm,yshift=-0.5812837616745371cm,line width=1.pt,dash pattern=on 3pt off 3pt] plot ({0.7764158285220606*(1+(\x)^2)/(1-(\x)^2)},{1.715603962768049*2*(\x)/(1-(\x)^2)});
\draw [samples=100,domain=-0.99:0.99,rotate around={155.6503365434604:(-0.6595553474635177,-0.5812837616745371)},xshift=-0.6595553474635177cm,yshift=-0.5812837616745371cm,line width=1.pt,dash pattern=on 3pt off 3pt] plot ({0.7764158285220606*(-1-(\x)^2)/(1-(\x)^2)},{1.715603962768049*(-2)*(\x)/(1-(\x)^2)});
\draw [samples=100,domain=-0.99:0.99,rotate around={155.6503365434604:(-0.6595553474635177,-0.5812837616745371)},xshift=-0.6595553474635177cm,yshift=-0.5812837616745371cm,line width=1.pt,dash pattern=on 3pt off 3pt] plot ({0.7764158285220606*(1+(\x)^2)/(1-(\x)^2)},{1.715603962768049*2*(\x)/(1-(\x)^2)});
\draw [samples=100,domain=-0.99:0.99,rotate around={155.6503365434604:(-0.6595553474635177,-0.5812837616745371)},xshift=-0.6595553474635177cm,yshift=-0.5812837616745371cm,line width=1.pt,dash pattern=on 3pt off 3pt] plot ({0.7764158285220606*(-1-(\x)^2)/(1-(\x)^2)},{1.715603962768049*(-2)*(\x)/(1-(\x)^2)});
\begin{scriptsize}
\draw[color=black] (1.4049411651141237,1.857688743224384) node {$\mathcal{P}$};
\draw [fill=black] (0.,0.) circle (1.0pt);
\draw[color=black] (0.1417952725502341,-0.016656774773645078) node {$F$};
\draw [fill=black] (-0.6595553474634155,0.750738310493664) circle (1.0pt);
\draw[color=black] (-0.5644798501736611,1.0020092676165882) node {$E$};
\draw[color=black] (-1.1213506200136556,2.2) node {$\mathcal{D}$};
\draw [fill=black] (0.22973359762054757,1.208084101062958) circle (1.0pt);
\draw[color=black] (0.3183640532312079,1.4366401123697543) node {$A$};
\draw[color=black] (-1.4,-0.7) node {$\mathcal{H}$};
\draw [fill=black] (-1.2743057168589236,-0.03798342962010152) circle (1.0pt);
\draw[color=black] (-1.501652609172676,-0.0302389886721815) node {$C'$};
\draw [fill=black] (-1.274982728172159,-0.03745528188956215) circle (1.0pt);
\draw[color=black] (-1.1,-0.35) node {$B=C$};
\end{scriptsize}
\end{tikzpicture}
    \caption{$C=C'$.}
    \label{fig.7(D)}
  \end{subfigure}
    \caption{The $\triangle ABC$ circumscribes $\mathcal{P}$ if and only if $C=C'$.}
  \label{fig.7}
\end{figure}

\begin{proposition}\label{prop.3.1} Let $\mathcal{C}_1$ and $\mathcal{C}_2$ be two nowhere tangential smooth conics. Suppose the pair of tangents from $A\in \mathcal{C}_1$ to $\mathcal{C}_2$ intersect $\mathcal{C}_1$ at $B$ and $C$, and $A,B,C$ are distinct. Then $\triangle ABC$ is inscribed in $\mathcal{C}_1$ and circumscribed about $\mathcal{C}_2$ non-trivially if and only if 
$$S_{BB}S_{CC}=S^2_{BC}.$$
\end{proposition}
\begin{proof}
Since $\triangle ABC$ is non-trivial,  $A$, $B$, and $C$ are distinct. By construction, $AB$ and $AC$ are the tangents from $A \in \mathcal{C}_1$ to $\mathcal{C}_2$. Thus, $C$ lies on one of the tangents to $\mathcal{C}_2$ from $B$ if and only if $S_{BB}S_{CC}=S_{BC}^2$ (see eq. \eqref{eq.2.0.3}). This concludes the proof.
\end{proof}

\begin{lemma}\label{lemm.3.1}
Let the pair of tangents from $A \in \mathcal{D}(E)$ to $\mathcal{P}(p)$ intersect $\mathcal{D}(E)$ at points $B$ and $C$, which are  not necessarily distinct.
Then
    \begin{eqnarray*}
        S_{BB}S_{CC}-S^2_{BC}&=&\frac{4p\,S_{AA}\,\mathcal{Q}(E)\,f(A, E, p)}{(x_A^2+y_A^2)^2},
\end{eqnarray*}
where
\begin{eqnarray}\label{eq.3.0.1}
       \mathcal{Q}(E)&:=&x_E^2+y_E^2-1,
\end{eqnarray}
and $f(A,E,p)$ was defined in eq. \eqref{eq.2.1.7b}.
\end{lemma}
\begin{proof}
Use Proposition \ref{prop.2.1} to calculate  $B=(x_B,y_B)$ and $C=(x_C,y_C)$. The rest follows from a straightforward calculation.
\end{proof}

\begin{remark}\label{rem.3.2}
The circle $\mathcal{D}(E)$ contains the focus of $\mathcal{P}(p)$ if and only if $\mathcal{Q}(E)=0$. (See eq. \eqref{eq.3.0.1}.)  
\end{remark}

\begin{lemma}\label{lemm.3.2}
Let  $A\in \mathcal{D}(E) \cap \mathcal{P}(p)$ and let the tangent at $A$ to $\mathcal{P}(p)$ intersect $\mathcal{D}(E)$ at $B$. Then
    \begin{eqnarray*}
        f(B,E,p)&=&\frac{f(A,E,p)\,\mathcal{Q}(E)}{(p+x_A)^2}.
\end{eqnarray*}
\end{lemma}
\begin{proof}
From the assumption  $A\in \mathcal{D}(E) \cap \mathcal{P}(p)$ we get $S_{AA}=0$. Then, after some calculations, we obtain
\begin{equation*}
    f(B,E,p)-\frac{f(A,E,p)\,\mathcal{Q}(E)}{(p+x_A)^2}=K(A,E,p)\left((x_A-x_E)^2+(y_A-y_E)^2-1\right),
\end{equation*}
where $K(A,E,p)$ is some function of the coordinates of $A$, $E$ and the parameter $p$, whose explicit form is nonessential for us. Now, the result follows from the assumption that $A\in \mathcal{D}(E)$.

In particular, if $A\in \mathcal{D}(E) \cap \mathcal{P}(p)$ and the circle $\mathcal{D}(E)$ contains the focus of the parabola $\mathcal{P}(p)$, then $B\in \mathcal{D}(E)\cap \mathcal{H}(E,p)$.
\end{proof}

In the next lemma we want to prove that if the circumcircle of a triangle contains the focus of the parabola which is inscribed in the triangle, then the non-common tangent drawn from a point of common tangent to the parabola intersects the circle at a point of intersection of the circumcircle with the parabola.

\begin{lemma}\label{lemm.3.3}
Given a circle $\mathcal{D}(E)$ that contains the focus of a parabola $\mathcal{P}(p)$ and a point in the nonconvex complement of the parabola, let $A\in \mathcal{D}(E)$ belong to a common tangent of the circle and the parabola. Then the noncommom tangent from $A$ to the parabola intersects the circle at its point of intersection with the parabola, denoted by $B$.
\end{lemma}
\begin{proof}
Let $A\in \mathcal{D}(E)$. A direct calculation gives
    \begin{eqnarray*}
        S_{BB}&=&\frac{\left(\sqrt{S_{AA}}(x_Ax_E+y_Ay_E)+p (x_A y_E - x_E y_A) + y_A \left(x_A x_E + y_A y_E - \mathcal{Q}(E)\right)\right)^2}{\left(x_A^2+y_A^2\right)^2}.
\end{eqnarray*}
The assumption that the circle contains the focus of the parabola gives $\mathcal{Q}(E)=0$ (see Remark \ref{rem.3.2} and \eqref{eq.3.0.1}).  The assumption that $A$ belongs to a common tangent of the circle and the parabola gives $f(A,E,p)=0$, equivalently, from eq. \eqref{eq.2.1.7b}, we have
\begin{equation*}
    p=-2(x_A-x_E)(x_Ax_E+y_Ay_E).
\end{equation*}
Plugging $p$, we obtain $S_{BB}=0$; which was to be proved.
\end{proof}

\begin{remark}\label{rem.3.3}
     Lemmas \ref{lemm.3.2} and \ref{lemm.3.3} establish a one-to-one correspondence of the points of intersection of the circle $\mathcal{D}$ and parabola $\mathcal{P}$ with the points of intersection of the circle $\mathcal{D}$ and the hyperbola $\mathcal{H}$, when $\mathcal{D}$ contains the focus of the parabola $\mathcal{P}$ and a point in the nonconvex complement of the parabola. See Figure \ref{fig.8}.
\end{remark}

\begin{figure}
    \centering
\definecolor{uuuuuu}{rgb}{0.26666666666666666,0.26666666666666666,0.26666666666666666}
\definecolor{ffqqqq}{rgb}{1.,0.,0.}
\definecolor{sqsqsq}{rgb}{0.12549019607843137,0.12549019607843137,0.12549019607843137}
\definecolor{qqqqff}{rgb}{0.,0.,1.}
\definecolor{ttzzqq}{rgb}{0.2,0.6,0.}
\definecolor{xfqqff}{rgb}{0.4980392156862745,0.,1.}
\begin{tikzpicture}[scale=1]
\clip(-4,-4) rectangle (3,5);
\draw [samples=100,rotate around={-90.:(-0.75,0.)},xshift=-0.75cm,yshift=0.cm,line width=1.pt,color=qqqqff,domain=-9.0:9.0)] plot (\x,{(\x)^2/2/1.5});
\draw [line width=1.pt] (-1.3535339570453924,1.4723945894780586) circle (2.cm);
\draw [samples=100,domain=-0.99:0.99,rotate around={156.29575850215014:(-1.3535339570453924,-1.2442684766482286)},xshift=-1.3535339570453924cm,yshift=-1.2442684766482286cm,line width=1.pt,dash pattern=on 5pt off 5pt,color=ffqqqq] plot ({1.3375938739216724*(1+(\x)^2)/(1-(\x)^2)},{3.046511111211065*2*(\x)/(1-(\x)^2)});
\draw [line width=1.pt,color=magenta] (-1.7572680960896627,3.431220448244014)-- (-0.6925317219115247,-0.41521661125902193);
\draw [line width=1.pt,color=brown] (-2.391998103889617,-0.2368719829789989)-- (0.5769282936882578,1.9951904372928349);
\begin{scriptsize}
\draw [fill=sqsqsq] (-1.3535339570456164,1.4723945894781179) circle (1.5pt);
\draw[color=sqsqsq] (-1.6,1.4723945894781179) node {$E$};
\draw [fill=magenta] (-1.7572680960896627,3.431220448244014) circle (2.0pt);
\draw[color=uuuuuu] (-1.95,3.5860868080561428) node {$X$};
\draw [fill=brown] (-2.391998103889617,-0.2368719829789989) circle (2.0pt);
\draw[color=uuuuuu] (-2.7,-0.4) node {$X'$};
\draw [fill=brown] (0.5769282936882578,1.9951904372928349) circle (2.0pt);
\draw[color=uuuuuu] (0.8,1.9) node {$Y'$};
\draw [fill=magenta] (-0.6925317219115247,-0.41521661125902193) circle (2.0pt);
\draw[color=uuuuuu] (-0.8,-0.7115323747460578) node {$Y$};
\draw [fill=black] (0.,0.) circle (1.5pt);
\draw[color=black] (0.15,-0.1) node {$F$};
\draw[color=uuuuuu] (-2.5,-1) node {$\mathcal{H}$};
\draw[color=uuuuuu] (1.5,2.45) node {$\mathcal{P}$};
\draw[color=uuuuuu] (-1.2,3.8) node {$\mathcal{D}$};
\end{scriptsize}
\end{tikzpicture}
\begin{tikzpicture}[scale=0.7]
\clip(-1,-6) rectangle (9,6);
\draw [samples=100,rotate around={-90.:(-0.75,0.)},xshift=-0.75cm,yshift=0.cm,line width=1.pt,color=qqqqff,domain=-12.0:12.0)] plot (\x,{(\x)^2/2/1.5});
\draw [line width=1.pt] (3.980707962075464,0.392382620243712) circle (4.cm);
\draw [samples=100,domain=-0.5:0.5,rotate around={-87.1852398265373:(3.980707962075464,-40.384143083322336)},xshift=3.980707962075464cm,yshift=-40.384143083322336cm,line width=1.pt,dash pattern=on 4pt off 4pt,color=ffqqqq] plot ({35.27090209135132*(-1-(\x)^2)/(1-(\x)^2)},{1.7341430167261291*(-2)*(\x)/(1-(\x)^2)});
\draw [line width=1.pt,color=magenta] (3.1332835022132772,4.301586092158071)-- (0.19484891972445872,1.6836112256614875);
\draw [line width=1.pt,color=brown] (0.6477858330622434,2.604085602548068)-- (5.165844258026346,4.2127820705656065);
\draw [line width=1.pt,color=ttzzqq] (2.712751434158585,-3.4013340591240255)-- (1.0359130558339935,-2.314679063607303);
\draw [line width=1.pt,color=xfqqff] (1.4675951547168018,-2.719572395094695)-- (3.526225614717059,-3.581714232619791);
\begin{scriptsize}
\draw [fill=sqsqsq] (3.9807079620754293,0.39238262024366377) circle (2.0pt);
\draw[color=sqsqsq] (4.100555557414623,0.7266523185530258) node {$E$};
\draw [fill=magenta] (0.19484891972445872,1.6836112256614875) circle (2.5pt);
\draw[color=uuuuuu] (-0.09197572341399168,1.8800561207023776) node {$Y$};
\draw [fill=magenta] (3.1332835022132772,4.301586092158071) circle (2.5pt);
\draw[color=uuuuuu] (3.4,4.607947652769892) node {$X$};
\draw [fill=brown] (0.6477858330622434,2.604085602548068) circle (2.5pt);
\draw[color=uuuuuu] (0.3,2.7222239762399996) node {$X'$};
\draw [fill=brown] (5.165844258026346,4.2127820705656065) circle (2.5pt);
\draw[color=uuuuuu] (5.3454993438615475,4.516407668472324) node {$Y'$};
\draw [fill=xfqqff] (1.4675951547168018,-2.719572395094695) circle (2.5pt);
\draw[color=uuuuuu] (1.1,-3) node {$X'$};
\draw [fill=ttzzqq] (2.712751434158585,-3.4013340591240255) circle (2.5pt);
\draw[color=uuuuuu] (2.4,-3.630650934011192) node {$X$};
\draw [fill=ttzzqq] (1.0359130558339935,-2.314679063607303) circle (2.5pt);
\draw[color=uuuuuu] (0.71,-2.6) node {$Y$};
\draw [fill=xfqqff] (3.526225614717059,-3.581714232619791) circle (2.5pt);
\draw[color=uuuuuu] (3.5879316453482426,-3.9) node {$Y'$};
\draw [fill=black] (0.,0.) circle (2.0pt);
\draw[color=black] (0.2,0.25) node {$F$};
\end{scriptsize}
\end{tikzpicture}
    \caption{One-to-one correspondence of points in $\mathcal{D}\cap \mathcal{H}$ and $\mathcal{D}\cap \mathcal{P}$.}
    \label{fig.8}
\end{figure}

Now we are ready to prove the following:
\begin{theorem}\label{thm.3.1}
Let a circle and a parabola be nowhere tangential and the circle has a nonempty intersection with the nonconvex complement of the parabola. Then there exists a triangle inscribed in a circle and circumscribed about a parabola if and only if the circle contains the focus of the parabola. Moreover, in such a case, every point on the circle in the nonconvex complement of the parabola and every point of intersection of the circle with the parabola is a vertex of such a triangle.
\end{theorem}
\begin{proof}
Without loss of generality, consider a unit circle $\mathcal{D}(E)$ centered at $E$  containing the focus of the parabola $\mathcal{P}(p)$ and a point in the nonconvex complement of the parabola. Let $A$ be the one of such points or a point of intersection of the circle and parabola. Let the tangents from $A \in \mathcal{D}(E)$ to $\mathcal{P}(p)$ intersect $\mathcal{D}(E)$ at $B$ and $C$.

We consider the following cases: 
\begin{itemize}
    \item[Case I:]  $S_{AA} \neq 0$.\\
    If $f(A,E,p) \neq 0$, then $A\notin \mathcal{D}\cap (\mathcal{P}\cup \mathcal{H})$. Hence, $A,B$, and $C$ are distinct. Also by Lemma \ref{lemm.3.1}, $S_{BB}S_{CC}=S_{BC}^2$. It follows from the Proposition \ref{prop.3.1} that the $\triangle ABC$ circumscribes $\mathcal{P}$ non-trivially.

    If $f(A,E,p) = 0$, then $A\in \mathcal{D}\cap \mathcal{H}$ implies $B\in \mathcal{D}\cap \mathcal{P}$ (Lemma \ref{lemm.3.3}). Thus, $\triangle ABC$ circumscribes $\mathcal{P}$ trivially.
    \item[Case II:] $S_{AA} = 0$.\\
    By Lemma \ref{lemm.3.2}, it follows that $f(B,E,p)=0$. Hence, $\triangle ABC$ circumscribes $\mathcal{P}$ trivially.
\end{itemize}

To prove the converse, we assume that there exists a $\triangle ABC$ that circumscribes $\mathcal{P}$. We now consider the following cases: 
\begin{itemize}
    \item[Case I:]  $\triangle ABC$ is non-trivial.\\
    If $\mathcal{Q}(E) = 0$, the proof is completed. Thus, suppose on the way to the contradiction, that $\mathcal{Q}(E) \neq 0$. Now by Proposition \ref{prop.3.1}, $S_{BB}S_{CC}=S_{BC}^2$ which is equivalent to $S_{AA}=0$ or $f(A,E,p)=0$ (but not both), by Lemma \ref{lemm.3.1}. Therefore, the $\triangle ABC$ is trivial. This, however, contradicts our hypothesis that the $\triangle ABC$ is non-trivial.
    \item[Case II:]  $\triangle ABC$ is trivial.\\
    After relabeling the points if needed, we get $S_{AA}=0$ and $f(B,E,p)=0$, see Remark \ref{rem.3.1}. Since by assumption the conics are nowhere tangent, we have $f(A,E,p)\neq 0$. Thus, it follows from the Lemma \ref{lemm.3.2} that $\mathcal{Q}(E)=0$.
\end{itemize}
See Figure \ref{fig.7} to accompany the proof. This concludes the proof. 
\end{proof}

\begin{remark}\label{rem.3.4}
An alternative proof of Case II of the converse part follows from Case II of the direct part and the following lemma.

\begin{lemma}\label{lemm.3.4} Given a parabola $\mathcal P$ with the directrix $\ell$ and the focus $F$ and a circle $\mathcal D$ which transversally intersects $\mathcal P$ at a point $A$. Denote by $T_1$ the orthogonal projection of $A$ at $\ell$ and by $B$ the other intersection point of the tangent to  $\mathcal P$ at $A$ by $\mathcal{D}$. Denote by $T_2$ another intersection point of the circle with the center $B$ and the radius $BF$ with $\ell$. 
Then $B$ is a point of a common tangent of $\mathcal D$ and $\mathcal P$ if and only if $T_2F$ is parallel to $BE$, where $E$ is the center of $\mathcal D$ (Figure \ref{fig.9}.)
\end{lemma}
\begin{proof}
By the focal property of parabola, the tangent at $A$ to parabola $\mathcal P$ is the bisector of the angle $T_1AF$. Another tangent from $B$ to $\mathcal P$ is the bisector of the angle $T_2BF$ and denote by $D$ its point of tangency with the parabola $\mathcal P$. Thus, $FBT_2D$ is a kite. Thus, its diagonals are orthogonal to each other: $BD$ is orthogonal to $FT_2$. Then $B$ is a point of a common tangent of the parabola $\mathcal P$ and a circle $\mathcal D$ if and only if $EB$ is orthogonal to $BD$, where $E$ is the center of the circle. 

Thus, $B$ is a point of a common tangent of the parabola $\mathcal P$ and a circle $\mathcal D$ if and only if $EB$ is parallel $FT_2$.

\begin{figure}
    \centering
\begin{tikzpicture}[scale=2]
\clip(-3,-2) rectangle (2,2);
\draw [samples=100,rotate around={-90.:(-0.5,0.)},xshift=-0.5cm,yshift=0.cm,line width=1.pt,domain=-6.0:6.0)] plot (\x,{(\x)^2/2/1.0});
\draw [line width=1.pt] (-0.7759458332394195,0.6364260830945299) circle (1.cm);
\draw [line width=1.pt] (-1.,-2.7420855391163386) -- (-1.,3.7538666145656188);
\draw [line width=1.pt,domain=-3.5331824343221037:6.758818769529621] plot(\x,{(--1.1031421960431422-0.*\x)/1.});
\draw [line width=1.pt] (-1.3269602965285907,-0.19806960967215348) circle (1.3416613577348515cm);
\draw [line width=1.pt,domain=-3.5331824343221037:6.758818769529621] plot(\x,{(-0.--1.4992814153874494*\x)/1.});
\draw [line width=1.pt,domain=-3.5331824343221037:6.758818769529621] plot(\x,{(--0.9982034322581126--0.8344956927666833*\x)/0.5510144632891713});
\draw [line width=1.pt,domain=-3.5331824343221037:6.758818769529621] plot(\x,{(-2.1130669740648367-1.3012118057152957*\x)/1.9508826777916872});
\draw [line width=1.pt,domain=-3.5331824343221037:6.758818769529621] plot(\x,{(--1.4423429978510331--1.301211805715296*\x)/1.435421648874034});
\draw [line width=1.pt] (0.,0.)-- (-1.3269602965285907,-0.19806960967215348);
\draw [line width=1.pt] (-1.3269602965285907,-0.19806960967215348)-- (-1.,-1.4992814153874494);
\draw [line width=1.pt] (-1.,-1.4992814153874494)-- (0.6239223812630965,-1.4992814153874492);
\draw [line width=1.pt] (0.6239223812630965,-1.4992814153874492)-- (0.,0.);
\begin{scriptsize}
\draw[color=black] (1.5649997202663493,1.918373905229664) node {$\mathcal{P}$};
\draw [fill=black] (-0.7759458332394195,0.6364260830945299) circle (1.0pt);
\draw[color=black] (-0.6935023309005411,0.6530242218397471) node {$E$};
\draw[color=black] (-1.3261771725955005,1.5946798001764295) node {$\mathcal{D}$};
\draw[color=black] (-0.8774194360444247,1.9) node {$\ell$};
\draw [fill=black] (0.,0.) circle (1.0pt);
\draw[color=black] (0.09366287911528054,0.012992695939033287) node {$F$};
\draw [fill=black] (0.10846135234544318,1.1031421960431422) circle (1.0pt);
\draw[color=black] (0.12308961593830192,1.2709856951231948) node {$A$};
\draw[color=black] (-3.452258908058795,1.050285168950535) node {$f$};
\draw [fill=black] (-1.,1.1031421960431422) circle (1.0pt);
\draw[color=black] (-0.85,1.245237300403051) node {$T_1$};
\draw[color=black] (-3.452258908058795,-2.002738776437927) node {$g$};
\draw [fill=black] (-1.3269602965285907,-0.19806960967215348) circle (1.0pt);
\draw[color=black] (-1.4659541725048522,-0.19) node {$B$};
\draw[color=black] (-2.245762698314919,0.9105081690411837) node {$d$};
\draw [fill=black] (-1.,-1.4992814153874494) circle (1.0pt);
\draw[color=black] (-0.9105245149703237,-1.63) node {$T_2$};
\draw [fill=black] (0.6239223812630965,-1.4992814153874492) circle (1.0pt);
\draw[color=black] (0.6748409313699526,-1.4215607241832562) node {$D$};
\end{scriptsize}
\end{tikzpicture}
    \caption{Lemma \ref{lemm.3.3}.}
    \label{fig.9}
\end{figure}
\end{proof}
\end{remark}

We now provide an alternative proof of the Lemma \ref{lemm.3.3}. Here we use  linear pencils of conics: $\lambda \mathcal{D}+\mathcal{P}=0$ and $\mu \mathcal{D}+\mathcal{H}=0$. 
We denote the discriminant of a polynomial $q\in \mathbb{R}[x]$  by $\mathrm{Disc}_{x}(q)$. 

\begin{lemma}\label{lemm.3.5}
Let 
\begin{eqnarray*}
    \mathcal{D}(E)&:=& x^2+y^2-2(x_Ax_E+y_Ay_E)=0,\\
    \mathcal{P}(p)&:=& y^2=2px+p^2,\\
    \mathcal{H}(E,p)&:=& 2x_E x^2 + 2y_E x y -2(2 x_E^2 -y_E^2 -1) x - 2x_E y_E y + p + 2x_E^3 - 2x_E y_E^2 - 2x_E=0.
\end{eqnarray*}
Then
\begin{eqnarray*}
    \mathrm{Disc}_{\lambda}\left(\det (\lambda \mathcal{D}+\mathcal{P})\right)= p^2\,\mathrm{Disc}_{\mu}\left(\det (\mu \mathcal{D}+\mathcal{H})\right),
\end{eqnarray*}
where $\mathcal{D}, \mathcal{P}$, $\mathcal{H}$  are the matrices associated with the circle $\mathcal{D}(E)$, the parabola $\mathcal{P}(p)$, and the conic $\mathcal{H}(E,p)$.  
\end{lemma}
The proof of the Lemma \ref{lemm.3.5} follows from a direct calculation.

\begin{corollary}\label{cor.3.1}
    If $\mathcal{D}(E)$ contains the focus of a parabola $\mathcal{P}(p)$ and a point in the nonconvex complement of the parabola, then the pencils: $\lambda \mathcal{D}+\mathcal{P}=0$ and $\mu \mathcal{D}+\mathcal{H}=0$ have the same number of real degenerate conics.
\end{corollary}

\noindent
\textit{An alternative proof of Lemma \ref{lemm.3.3}.} By assumption, $\mathcal{D}(E)$ has at least two real intersection points with the parabola $\mathcal{P}$. It follows from the Lemma \ref{lemm.3.2} that for every $Y\in \mathcal{D}\cap \mathcal{P}$, there exists an $X \in \mathcal{D}\cap \mathcal{H}$.

By Corollary \ref{cor.3.1}, the numbers of real degenerate conics in two pencils are the same: one or three. If this number is one, then the circle and the parabola intersect in two real points as well as the circle and the hyperbola. If this number is three, then the circle and the parabola intersect in four real points, so as the circle and the hyperbola.    
 Therefore, we have an one-to-one correspondence between the points $X\in \mathcal{D}\cap \mathcal{H}$ and $Y \in \mathcal{D}\cap \mathcal{P}$ via the tangents (see Figure \ref{fig.8}).  $\square$

The next theorem gives the formulas for the coordinates of the orthocenter, centroid and the nine-point center of the triangle that circumscribes a parabola in terms of one of the vertices when it varies along the circumcircle. 
\begin{theorem}\label{thm.3.2}
 Let a circle and a parabola form a 3-Poncelet pair. Then the orthocenter of an associated Poncelet triangle lies on the directrix of the parabola. The line segments traced out by the centroids and the nine-point centers of the associated Poncelet triangles are parallel to the directrix of the parabola. In particular, for the 3-Poncelet pair $(\mathcal{D}(E), \mathcal{P}(p))$, the coordinates of the orthocenter $O$, centroid $G$ and the nine-point center $N$, as functions of the parameter $p$ of the parabola and the coordinates of $E$ and the vertex $A$, are given by:
\begin{subequations}\label{eq.3.0.2}
\begin{eqnarray}
  O&=& \left(-p,\,y_A+\frac{(x_A+p)(x_Ey_A-x_Ay_E)}{(x_Ax_E+y_Ay_E)}\right),\label{eq.3.0.2a}\\
    G&=&\frac{1}{3}\left(2x_E-p,\,y_A+2y_E+\frac{(x_A+p)(x_Ey_A-x_Ay_E)}{(x_Ax_E+y_Ay_E)}\right),\label{eq.3.0.2b}\\
    N&=&\frac{1}{2}\left(x_E-p,\,y_A+y_E+\frac{(x_A+p)(x_Ey_A-x_Ay_E)}{(x_Ax_E+y_Ay_E)}\right)\label{eq.3.0.2c}.
\end{eqnarray}
\end{subequations}
\end{theorem}

\begin{proof}
Let $\triangle ABC$ be a triangle that circumscribes the parabola $\mathcal{P}(p)$ and $\mathcal{D}$ be the circumcircle of $\triangle ABC$ with $E$ be the circumcenter.

The coordinates of the orthocenter $O$ of the $\triangle ABC$ can be calculated by the following formulas: 
\begin{eqnarray*}
    x_O&=& \frac{m_{AB} m_{AC} (y_B - y_C) + m_{AB} x_B - m_{AC} x_C}{m_{AB} - m_{AC}},\\
    y_O&=& \frac{m_{AB} y_C - m_{AC} y_B - x_B + x_C}{m_{AB} - m_{AC}}.
\end{eqnarray*}

\begin{figure}
    \centering
\definecolor{aqaqaq}{rgb}{0.6274509803921569,0.6274509803921569,0.6274509803921569}
\definecolor{wewdxt}{rgb}{0.43137254901960786,0.42745098039215684,0.45098039215686275}
\begin{tikzpicture}[scale=3]
\clip(-2,-1.4) rectangle (1,1.8);
\draw [color=black,samples=100,rotate around={-90.:(-0.5,0.)},xshift=-0.5cm,yshift=0.cm,line width=1.pt,domain=-4.0:4.0)] plot (\x,{(\x)^2/2/1.0});
\draw [color=black,line width=1.pt] (-0.8328925983487512,0.553434657042605) circle (1.cm);
\draw [color=black,line width=1.pt] (-1.7818458965934851,0.8688512668627617)-- (-0.3202987184235915,1.4120658434452833);
\draw [color=black,line width=1.pt] (-0.3202987184235915,1.4120658434452833)-- (-0.5636405816804256,-0.40963509719124064);
\draw [color=black,line width=1.pt] (-0.5636405816804256,-0.40963509719124064)-- (-1.7818458965934851,0.8688512668627617);
\draw [color=black,line width=1.pt] (-0.9164462991743755,0.6589236780370997) circle (0.5cm);
\draw [color=black,line width=1.pt] (-1.,-1.4481772138718243) -- (-1.,2.031757878340348);
\draw [color=black,line width=1.pt,domain=-2.2832289818024645:0.2] plot(\x,{(--0.08323902206460176--0.2109780419889895*\x)/-0.1671074016512486});
\draw [color=black,line width=1.pt,dash pattern=on 3pt off 3pt] (-1.7818458965934851,0.8688512668627617)-- (-0.41721087284990666,0.6865640308319856);
\draw [color=black,line width=1.pt,dash pattern=on 3pt off 3pt] (-0.3202987184235915,1.4120658434452833)-- (-1.3576306904584459,0.42364441229677896);
\draw [color=black,line width=1.pt,dash pattern=on 3pt off 3pt] (-0.5636405816804256,-0.40963509719124064)-- (-1.129000380314067,1.1114949570538546);
\begin{scriptsize}
\draw [fill=black] (0.,0.) circle (0.6pt);
\draw[color=black] (0.03015935243144971,-0.043197757050035576) node {$F$};
\draw[color=black] (0.5,1.3) node {$\mathcal{P}$};
\draw [fill=black] (-0.8328925983487512,0.553434657042605) circle (0.6pt);
\draw[color=black] (-0.75,0.5558397673171787) node {$E$};
\draw[color=black] (-1.06,1.7) node {$\ell$};
\draw[color=black] (-1.368908549873558,1.5) node {$\mathcal{D}(E)$};
\draw [fill=black] (-1.7818458965934851,0.8688512668627617) circle (0.6pt);
\draw[color=black] (-1.8457739475606172,0.8947688929459972) node {$A$};
\draw [fill=black] (-0.3202987184235915,1.4120658434452833) circle (0.6pt);
\draw[color=black] (-0.2890645914747634,1.481983308279648) node {$B$};
\draw [fill=black] (-0.5636405816804256,-0.40963509719124064) circle (0.6pt);
\draw[color=black] (-0.5728192082802861,-0.46882968225831945) node {$C$};
\draw [fill=aqaqaq] (-0.41721087284990666,0.6865640308319856) circle (0.6pt);
\draw [fill=aqaqaq] (-1.3576306904584459,0.42364441229677896) circle (0.6pt);
\draw [fill=aqaqaq] (-1.129000380314067,1.1114949570538546) circle (0.6pt);
\draw [fill=black] (-1.,0.7644126990315945) circle (0.6pt);
\draw[color=black] (-0.9590407700433586,0.8711226748788703) node {$O$};
\draw [fill=aqaqaq] (-1.0510723075085382,1.1404585551540225) circle (0.6pt);
\draw [fill=aqaqaq] (-1.1727432391369554,0.22960808483576053) circle (0.6pt);
\draw [fill=aqaqaq] (-0.44196965005200856,0.5012153731270214) circle (0.6pt);
\draw [fill=black] (-0.9164462991745977,0.658923678037442) circle (0.6pt);
\draw[color=black] (-0.85,0.7) node {$N$};
\draw [fill=black] (-0.8885950655658341,0.623760671038935) circle (0.6pt);
\draw[color=black] (-0.81,0.6464836032411649) node {$G$};
\draw [fill=aqaqaq] (-0.7818202908402129,0.17738880092017695) circle (0.6pt);
\draw [fill=aqaqaq] (-1.3909229482967422,0.8166319829471781) circle (0.6pt);
\draw [fill=aqaqaq] (-0.6601493592117957,1.0882392712384386) circle (0.6pt);
\end{scriptsize}
\end{tikzpicture}
    \caption{Theorem \ref{thm.3.2}. $\triangle ABC$ with its circumcenter $E$,  orthocenter $O$, centroid $G$, and nine-point center $N$.}
    \label{fig.10}
\end{figure}

Now we calculate the coordinates of $B$ and $C$ by using \eqref{eq.2.1.5}, and use the eq. \eqref{eq.2.1.3} for $m_{AB},m_{AC}$ to find
\begin{subequations}\label{eq.3.0.3}
\begin{eqnarray}
    x_O&=& -p+\frac{(x_A+p) \mathcal{Q}(E)}{2(x_Ax_E+y_Ay_E)-\mathcal{Q}(E)},\label{eq.3.0.3a}\\
    y_O&=&y_A+\frac{2(x_A+p)(x_Ey_A-x_Ay_E)}{2(x_Ax_E+y_Ay_E)-\mathcal{Q}(E)}.\label{eq.3.0.3b}
\end{eqnarray}
\end{subequations}
Since the $\triangle ABC$ circumscribes the parabola $\mathcal{P}(p)$, the circumcircle $\mathcal{D}(E)$ contains the focus of the parabola $\mathcal{P}(p)$ (Theorem \ref{thm.3.1}.) Thus, by using $\mathcal{Q}(E)=0$ in eq. \eqref{eq.3.0.3} one obtains the coordinates of $O$.

Now the coordinates of an arbitrary point $X_t\in \mathbb{R}^2$ on the Euler line can be calculated as 
\begin{equation}\label{eq.3.0.4}
    X_t= \left((1-t)x_E+tx_O,(1-t)y_E+ty_O\right),\qquad t\in \mathbb{R}. 
\end{equation}
As, $E$ is fixed and the $x$-coordinate of $X_t$ is independent of the coordinates of $A$, for a given $t$, $X_t$ traces a line parallel to the directrix $\ell: x=-p$. 

From the property of Euler line, it follows that the centroid and the nine-point center are given by $G=X_{1/3}$ and $N=X_{1/2}$, respectively. Now by using the eqs. \eqref{eq.3.0.3} and \eqref{eq.3.0.4}, the coordinates of $G$ and $N$ can be easily obtained.

These points are shown in Figure \ref{fig.10}. This completes the proof.
\end{proof}

\begin{proposition}\label{prop.3.2}
    Given a parabola, a point $O$  on its directrix $\ell$, and a pair of its tangents $t_1, t_2$ that intersect at a point $A \notin \ell$. Denote by $s_1, s_2$ the lines through $O$ orthogonal to $t_1, t_2$, respectively, and $B=s_1\cap t_2$, $C=s_2\cap t_1$. Then, the $\triangle ABC$ circumscribes the parabola and has $O$ as its orthocenter.
\end{proposition}

\begin{proof} Two tangents to a parabola are perpendicular to each other if and only if they intersect at the directrix. This easily follows from eq. \eqref{eq.2.1.9b}. Since $A\notin \ell$, the pair of tangents $t_1, t_2$ from $A$ to the parabola are not perpendicular to each other. So, $s_1 \nparallel t_2$ and $s_2 \nparallel t_1$. Hence, $B$ and $C$ exist.

By construction, $O$ is the orthocenter of $\triangle ABC$. If the circumcircle of the $\triangle ABC$ has the radius $R$ and $E$ denotes its circumcenter, then it follows from the equation \eqref{eq.3.0.3a} (by scaling the radius from $1$ to $R$) and $x_A+p \neq 0$ that $x_O=-p$ if and only if $\mathcal{Q}(E,R):x_E^2+y_E^2-R^2=0$. Thus, the circumcircle contains the focus. Now, we see from Theorem \ref{thm.3.1} that $\triangle ABC$ circumscribes the parabola.
\end{proof}

In the next proposition we find the loci of the midpoints of the sides of an $n$-Poncelet polygon that circumscribes a parabola when the vertices vary along the circumcircle of the $n$-gon. To that end let us recall that the \textit{pedal curve} of a curve $\mathcal{C}$ with respect to a given point $P$, called the \textit{pedal point}, is the locus of points $X$ such that $PX \perp T$ where $T$ is a tangent passing through $X$ to the curve $\mathcal{C}$ (see, for example, \cite{Hilbert-Vossen1990}.)
\begin{proposition}\label{prop.3.3}
Let a circle $\mathcal{D}$ and a conic $\mathcal{S}$ form an $n$-Poncelet pair $(\mathcal{D},\mathcal{S})$. Then the midpoints of each side of the associated $n$-Poncelet polygons lie on the pedal curve of the conic $\mathcal{S}$ with the center of $\mathcal{D}$ as the pedal point.

The pedal curve self-intersect at the pedal point if and only if the pedal point lies in the nonconvex complement of the conic.
\end{proposition}
\begin{proof}
Let  $AB$ be a side of an arbitrary $n$-Poncelet polygon associated with the $n$-Poncelet pair $(\mathcal{D},\mathcal{S})$.
Since the perpendicular through the center $E$ to the tangent $AB$  bisects the chord $AB$ at its midpoint $M_{AB}$, we get that $M_{AB}$ lies on the pedal curve.

The pedal curve contains the pedal point if and only if there exists a tangent to the given curve passes through the pedal point. There exist two tangents $t_1, t_2$ through the pedal point if and only if the pedal point lies in the nonconvex complement of the conic. Thus, any variable tangent moving from $t_1$ to $t_2$ will make the pedal curve self-intersect at the pedal point.

\end{proof}
\begin{corollary}\label{cor.3.2}
 The pedal curve of the parabola $\mathcal{P}(p)$, defined in eq. \eqref{eq.1.0.2}, with the center $E$ of the circle $\mathcal{D}$ as the pedal point is given by:
\begin{eqnarray}\label{eq.3.0.5}
    &&2 x^3+2 x y^2+ \left(p-4 x_E\right)x^2-2y_E x y+ \left(p-2 x_E\right)y^2\nonumber \\
    &&+2 x_E \left(x_E-p\right)x+2y_E \left(x_E-p\right)y+p(x_E^2+y_E^2)=0.
    \end{eqnarray} 
\end{corollary}
\begin{proof}
To obtain the equation of the pedal curve with pedal point $E$, we use the following parametric equations for the parabola $\mathcal{P}(p)$:
\begin{eqnarray*}
    x &=& \frac{p}{2}(t^2-1),\\ 
    y &=& pt.
\end{eqnarray*}

Therefore, the parametric equations for the pedal curve are given by
\begin{subequations}\label{eq.3.0.6}
    \begin{eqnarray}
    x &=&-\frac{p}{2}+ \frac{t(t x_E + y_E)}{t^2 + 1},\label{eq.3.0.6a}\\ 
    y &=&\frac{p}{2}t+ \frac{t x_E + y_E}{t^2 + 1}.\label{eq.3.0.6b}
\end{eqnarray}
\end{subequations}
After eliminating $t$ from the eqs. \eqref{eq.3.0.6a} and \eqref{eq.3.0.6b}, we obtain the equation \eqref{eq.3.0.5}. 
\end{proof}

\begin{remark}
For $n=3$, using $x_E^2+y_E^2=1$, we reduce the equation \eqref{eq.3.0.5} to the following:
\begin{equation*}
    2 x^3+2 x y^2+ \left(p-4 x_E\right)x^2-2y_E x y+ \left(p-2 x_E\right)y^2+2 x_E \left(x_E-p\right)x+2y_E \left(x_E-p\right)y+p=0.
    \end{equation*}
The pedal curve of the parabola $\mathcal{P}(p)$ with $E$ as the pedal point for Poncelet triangles is shown in Figure \ref{fig.11}.
\end{remark}
\begin{figure}
    \centering
    \includegraphics[width=0.7\linewidth]{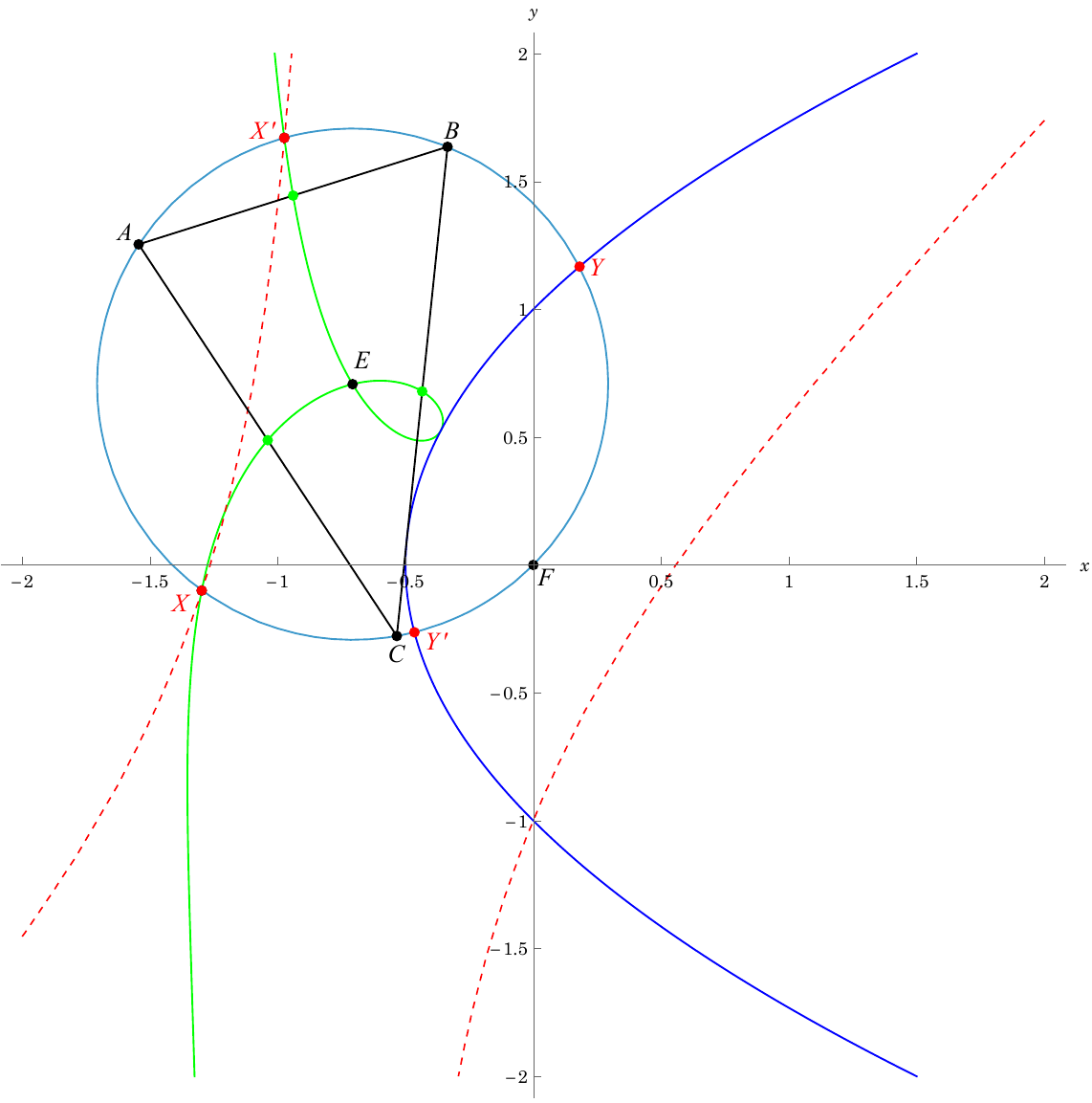}
    \includegraphics[width=0.7\linewidth]{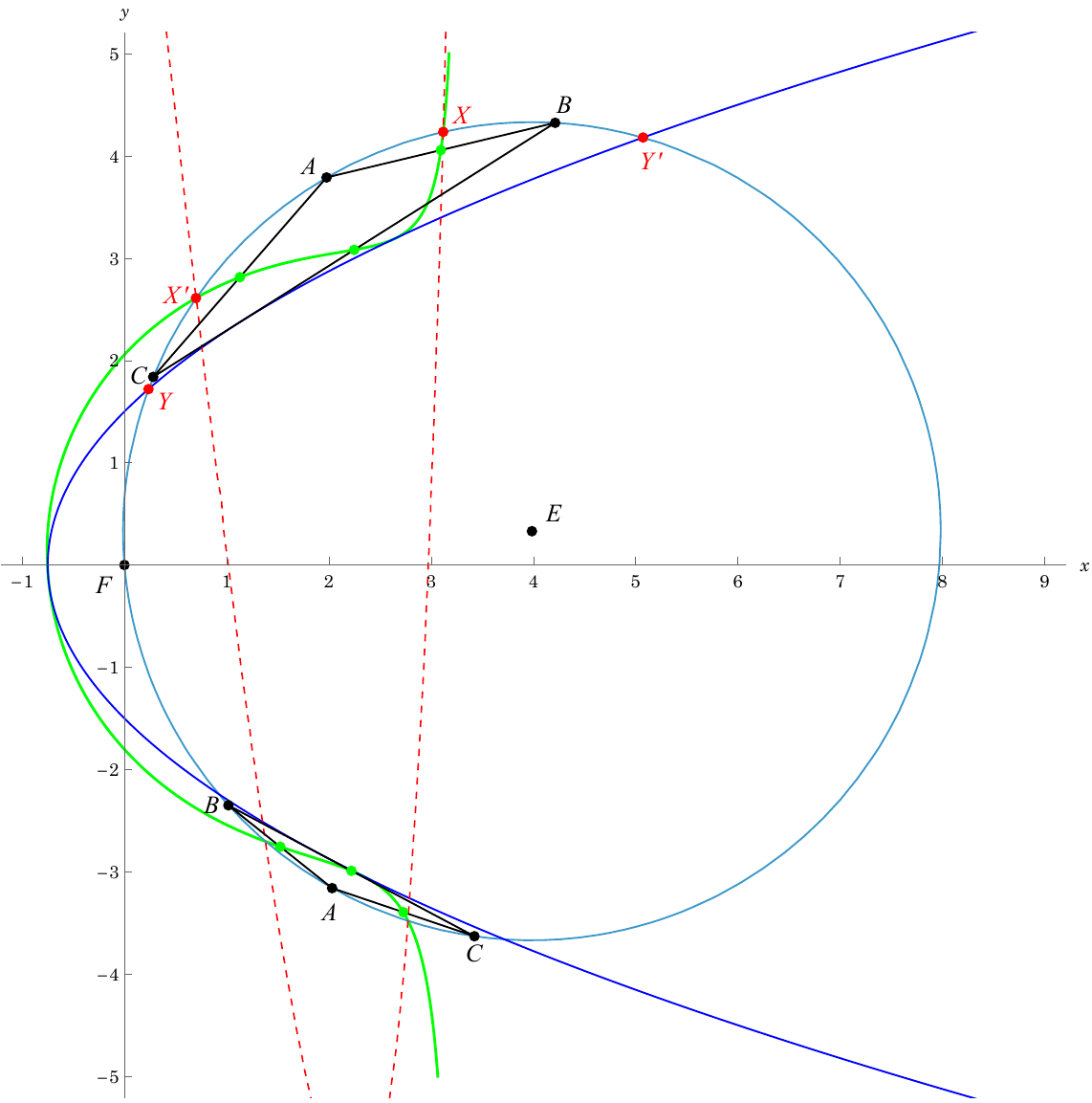}
        \caption{Pedal curve of $\mathcal{P}$ with pedal point $E$.}
    \label{fig.11}
\end{figure}

\noindent We observed in Theorem \ref{thm.3.2} that  any point of the Euler line of a triangle, that inscribes a parabola, traces a line segment parallel to the directrix of the parabola when a triangle vary. The next theorem singles out the triangles for which the extremities of the segment occur.

\begin{theorem}\label{thm.3.3}
Given a parabola $\mathcal{P}$ and a circle $\mathcal{D}$ containing the focus of $\mathcal{P}$ with nonempty intersection with the nonconvex complement of $\mathcal{P}$. Consider $Y,Y'\in \mathcal{D}\cap \mathcal{P}$ such that the arc $YY'$ of $\mathcal{D}$ lies in the nonconvex complement of $\mathcal{P}$. Denote by $X$ and $X'$ the intersection with $\mathcal{D}$ of the tangents to $\mathcal{P}$ at $Y$ and $Y'$, respectively. The orthocenters of the triangles circumscribed about $\mathcal{P}$ with vertices in the arc $YY'$ traces a line segment with the two endpoints being the orthocenters of triangles with a vertex at $X$ and $X'$, respectively.
\end{theorem}
\begin{proof} 
Consider a $\triangle ABC$ that circumscribes a parabola $\mathcal{P}(p)$, which is inscribed in $\mathcal{D}$ and denote the center of the circle by $E$. 

By Theorem \ref{thm.3.1}, $\mathcal{D}$ contains the focus $F=(0,0)$ of the parabola $\mathcal{P}(p)$. Clearly, $E \neq F$.

There are two geometrically possible situations, depending on the number of intersection points of the parabola and the circle, as presented in Figure \ref{fig.11}. The  following consideration applies equally to both of them.

Define $$L_E: x_E x+y_E y=0.$$
Since $L_E \cap \mathcal{D}=F$, we define $D_E \subseteq \mathbb{R}^2$ to be the complement of the line $L_E$ that contains the circle $\mathcal{D}$.

Since by assumption, the arc $YY'$ lies in the nonconvex complement of the parabola, by Lemma \ref{lemm.3.2} the points of intersections $X$ and $X'$ with the circle and the tangents at $Y$ and $Y'$ to the parabola, also lie on the arc $YY'$.

Now, define a real-valued function $h: D_E \to \mathbb{R}$ as follows:
    \begin{equation*}
        h(x,y)=y+\frac{(x+p)(x_E y-y_Ex)}{x_E x+y_Ey}.
    \end{equation*}
Note that $y_O=h(x_A,y_A)$ and $x_Ex+y_Ey\ne 0$ on $D_E$. We want to find the critical points of $h$, subject to the constraint:
\begin{equation*}
    g(x,y)=x^2+y^2-2(x_E x+y_E y)=0.
\end{equation*}
We introduce the Lagrange multiplier $\lambda$ and solve each of the following equations:
\begin{eqnarray*}
    h_x(x,y)+\lambda g_x(x,y)&=&0,\\
    h_y(x,y)+\lambda g_y(x,y)&=&0,
\end{eqnarray*}
for $\lambda$. Equating the two expressions obtained for $\lambda$, we get
\begin{equation*}
    p+2(x-x_E)(x_Ex+y_Ey)=0.
\end{equation*}
Thus, in order to be an extremal point, $A$ must be the point of a common tangent for the circle  $\mathcal{D}$ and the parabola $\mathcal{P}(p)$. Therefore, by Lemma \ref{lemm.3.3}, $A=X$ or $A=X'$. The proof is complete. 
\end{proof}

\begin{corollary}\label{cor.3.3}
    The length of the line segment---traced out by an arbitrary point $X_t$ on the Euler line as $A$ varies along the arc $YY'$---is the distance between the points $X_t$ at $A=X$ and $A=X'$.
\end{corollary}

\begin{remark}\label{rem.3.5}    
From the property of the Euler line, it is also easy to see that the length of the line segment traced out by the orthocenter, as $A$ varies along the circumcircle, is three times of the length of the line segment traced out by the centroid.
\end{remark}

\section{Quadrilaterals inscribed in a circle circumscribed about a parabola}\label{sec.4}

\subsection{The focus of the parabola coincides with the center of the circle}\label{sec.4.1}

Now we are going to study $4$-Poncelet pairs of a circle and a parabola, when the center of the circle coincides with the focus of the parabola.

Let us recall that an \textit{antiparallelogram} is a self-intersecting quadrilateral $ABCD$ with two pairs of congruent opposite sides, $AB \cong CD$ and $BC \cong AD$. An antiparallelogram $ABCD$ can be obtained from an isosceles trapezoid $ABDC$. 
Antiparallelograms are also known as \textit{Darboux butterflies}, see, for example, \cite{Boretal.2020}. Antiparallelograms can be  characterized as nonconvex $4$-periodic trajectories of elliptical billiards \cite{Dragovic-Radnovic2025}.

\begin{theorem}\label{thm.4.1}
Let $\mathcal{P}$ be a parabola with the focus $F$ and axis $\mathscr{a}$, and $\mathcal{D}(F)$ be a circle centered at $F$. For a Poncelet quadrilateral $ABCD$ inscribed in $\mathcal{D}$ and circumscribed about $\mathcal{P}$ (Figure \ref{fig.12}), denote by $G \in AD \cap BC$ and $H \in AB \cap CD$ the points of intersections of the pairs of opposite sides of the quadrilateral $ABCD$. It holds:
\begin{itemize}
    \item[(a)] The quadrilateral $ABCD$ is an antiparallelogram (Darboux butterfly).
    \item[(b)] The points $G$ and $H$ lie on the axis of $\mathcal{P}$. 
    \item[(c)] The diagonals $AC$ and $BD$ of the quadrilateral $ABCD$ are parallel, and both are perpendicular to the axis of the parabola. Thus, $ABDC$ is an isosceles trapezoid.
    \item[(d)] The midline $\mathscr{m}$ of the isosceles trapezoid $ABDC$ is tangent to the parabola at its vertex. 
\end{itemize}
\end{theorem}

\begin{remark}\label{rem.4.1}
Thus, the midline $\mathscr{m}$ is independent of the choice of the 4-Poncelet polygon $ABCD$ and the radius of the circle $\mathcal{D}$ and it is parallel to the directrix $\ell$.
\end{remark}

\begin{figure}
\centering
\definecolor{xdxdff}{rgb}{0.49019607843137253,0.49019607843137253,1.}
\definecolor{ududff}{rgb}{0.30196078431372547,0.30196078431372547,1.}
\begin{tikzpicture}[scale=2.0]
\clip(-3,-2) rectangle (3,2);
\draw [line width=1.pt] (0.,0.) circle (1.cm);
\draw [samples=100,rotate around={-90.:(-0.25,0.)},xshift=-0.25cm,yshift=0.cm,line width=1.pt,domain=-4.0:4.0)] plot (\x,{(\x)^2/2/0.5});
\draw [line width=1.pt] (-0.8452364011647343,0.5343925768066846)-- (0.3452364011647343,0.9385157576252103);
\draw [line width=1.pt] (-0.8452364011647344,-0.5343925768066846)-- (0.34523640116473436,-0.9385157576252103);
\draw [line width=1.pt] (0.34523640116473436,-0.9385157576252103)-- (-0.8452364011647343,0.5343925768066846);
\draw [line width=1.pt] (0.3452364011647343,0.9385157576252103)-- (-0.8452364011647344,-0.5343925768066846);
\draw [line width=1.pt,dash pattern=on 4pt off 4pt,domain=-4.063739269471895:4.749517095014842] plot(\x,{(--2.2707003604256735--0.9385157576252103*\x)/2.764695362803673});
\draw [line width=1.pt,dash pattern=on 4pt off 4pt,domain=-4.063739269471895:4.749517095014842] plot(\x,{(-2.270700360425674-0.9385157576252103*\x)/2.7646953628036735});
\draw [line width=1.pt,dash pattern=on 2pt off 2pt] (-0.8452364011647343,0.5343925768066846)--(-0.8452364011647344,-0.5343925768066846);
\draw [line width=1.pt,dash pattern=on 2pt off 2pt] (0.34523640116473436,-0.9385157576252103)--(0.3452364011647343,0.9385157576252103);
\begin{scriptsize}
\draw [fill=black] (0.,0.) circle (1.0pt);
\draw[color=black] (0.1,0.17) node {$F$};
\draw [fill=black] (-0.8452364011647343,0.5343925768066846) circle (1.0pt);
\draw[color=black] (-0.85,0.72) node {$A$};
\draw [fill=black] (0.3452364011647343,0.9385157576252103) circle (1.0pt);
\draw[color=black] (0.41,1.09) node {$B$};
\draw [fill=black] (-0.8452364011647344,-0.5343925768066846) circle (1.0pt);
\draw[color=black] (-0.98,-0.6) node {$C$};
\draw [fill=black] (0.34523640116473436,-0.9385157576252103) circle (1.0pt);
\draw[color=black] (0.36464180979875993,-1.05502879024742) node {$D$};
\draw [fill=black] (-0.4133155452748829,0.) circle (1.0pt);
\draw[color=black] (-0.423,0.1959904769200991) node {$G$};
\draw [fill=black] (-2.42,0.) circle (1.0pt);
\draw[color=black] (-2.42,0.2) node {$H$};
\draw[line width=0.2mm,color=black] (-2.8,0) -- (1.5,0);
\end{scriptsize}
\end{tikzpicture}
\caption{An antiparallelogram (Darboux butterfly) $ABCD$ inscribed in a circle and circumscribed about a parabola.}\label{fig.12}
\end{figure}

In order to prove the theorem, we will derive some additional properties of parabolas, circles centered at the focus and related $4$-Poncelet polygons.

\begin{lemma}\label{lemm.4.1}
Given a circle with radius $R$ centered at the focus of a parabola $\mathcal{P}$, and a chord $AB$ of the circle tangent to the parabola, there exists a point (say $T$) on the directrix $\ell$ of the parabola such that $A$ and $B$ lie on the circle centered at $T$ with radius $R$.
\end{lemma}

\begin{corollary}\label{cor.4.1} 
Given a circle with radius $R$ centered at the focus of a parabola $\mathcal{P}$, and a $4$-Poncelet polygon $ABCD$ inscribed in the circle and circumscribed about the parabola, then for each side (say $AB$) of the Poncelet polygon, there exists a point (say $T$) on the directrix $\ell$ of the parabola such that $A$ and $B$ lie on the circle centered at $T$ with radius $R$.
\end{corollary}

\begin{proof} To prove the Lemma, denote the point of contact the parabola $\mathcal{P}$ with the tangent $AB$ by $L$, its projection to the directrix $\ell$ by $T$, and the intersection of the line $AB$ with the axis of the parabola $\mathcal{P}$ by $G$ as shown in the Figure \ref{fig.13}. Using both the defining and focal properties of parabolas (see Lemma \ref{lemm.1.1} and the statement above it in Introduction), we get that the quadrilateral $FLTG$ is a rhombus. So, its diagonals $FT$ and $GL$ are orthogonal bisectors of each other. Thus, $FT \perp AB$. Since $\overline{AB}$ is a chord of a circle centered at $F$, $AB$ and $FT$ are orthogonal bisectors of each other. This means that $AFBT$ is a rhombus as well. 

Thus, $|TA|=|TB|=R$ and $T$ belongs to the directrix $\ell$.
\end{proof}

The converse is also true.
\begin{lemma}\label{lemm.4.2}
Given a circle with radius $R$ centered at the focus of a parabola $\mathcal{P}$, and a point (say $T$) on the directrix $\ell$ of the parabola such that two points of the circle $A$ and $B$ lie on the circle centered at $T$ with radius $R$. Then, the chord $AB$ belongs to a  tangent to the parabola.
\end{lemma}

\begin{proof}
Denote by $F$ the focus of the parabola.
The quadrilateral $AFBT$ is a rhombus by construction. Thus, its diagonals $TF$ and $AB$ are the orthogonal bisectors of each other. 
Denote by $L$ the intersection of the line through $T$ orthogonal to the 
directrix $\ell$ with the parabola. The tangent to the parabola at $L$ intersects the axis of the parabola at the point $G$. By the defining and the focal properties of parabola (see Lemma \ref{lemm.1.1} and the statement above it in Introduction), the quadrilateral $GTLF$ is a rhombus and its diagonals $TF$ and $GL$ are orthogonal bisectors of each other. Thus, $AB$ belongs to the tangent of the parabola at $L$.
\end{proof}

\begin{figure}
\centering
\definecolor{xdxdff}{rgb}{0.49019607843137253,0.49019607843137253,1.}
\definecolor{ududff}{rgb}{0.30196078431372547,0.30196078431372547,1.}
\definecolor{uququq}{rgb}{0.25098039215686274,0.25098039215686274,0.25098039215686274}
\begin{tikzpicture}[scale=2.0]
\clip(-3,-2) rectangle (3,3);
\draw [line width=1.pt] (0.,0.) circle (1.cm);
\draw [samples=100,rotate around={-90.:(-0.25,0.)},xshift=-0.25cm,yshift=0.cm,line width=1.pt,domain=-4.0:4.0)] plot (\x,{(\x)^2/2/0.5});
\draw [line width=1.pt] (-0.5,1.4897) circle (1.cm);
\draw [line width=1.pt] (-0.8452364011647343,0.5343925768066846)-- (0.3452364011647343,0.9385157576252103);
\draw [line width=1.pt] (-0.8452364011647344,-0.5343925768066846)-- (0.34523640116473436,-0.9385157576252103);
\draw [line width=1.pt] (0.34523640116473436,-0.9385157576252103)-- (-0.8452364011647343,0.5343925768066846);
\draw [line width=1.pt] (0.3452364011647343,0.9385157576252103)-- (-0.8452364011647344,-0.5343925768066846);
\draw [line width=1.pt,color=blue] (-0.5,1.4897)-- (-0.8452364011647343,0.5343925768066846);
\draw [line width=1.pt,color=blue] (-0.5,1.4897)-- (0.3452364011647343,0.9385157576252103);
\draw [line width=1.pt,dash pattern=on 4pt off 4pt,domain=-4.063739269471895:4.749517095014842] plot(\x,{(--2.2707003604256735--0.9385157576252103*\x)/2.764695362803673});
\draw [line width=1.pt,dash pattern=on 4pt off 4pt,domain=-4.063739269471895:4.749517095014842] plot(\x,{(-2.270700360425674-0.9385157576252103*\x)/2.7646953628036735});
\draw [line width=1.pt] (-0.5,-1.8)-- (-0.5,2.9);
\draw [line width=1.pt] (0,0)-- (1.9691, 1.4897);
\draw [line width=1.pt] (-2.42,0)-- (-0.5,1.4897);
\draw [line width=1.pt] (1.9691, 1.4897)-- (-0.5,1.4897);
\draw [line width=1.pt] (0,0)-- (-0.8452364011647343,0.5343925768066846);
\draw [line width=1.pt] (0,0)-- (0.3452364011647343,0.9385157576252103);
\draw [line width=1.pt] (0,0)-- (-0.5,1.4897);
\begin{scriptsize}
\draw [fill=black] (0.,0.) circle (1.0pt);
\draw[color=black] (0,-0.17) node {$F$};
\draw [fill=black] (-0.83648, 0.54799) circle (1.0pt);
\draw[color=black] (-0.87,0.72) node {$A$};
\draw [fill=black] (0.33648, 0.94169) circle (1.0pt);
\draw[color=black] (0.32,1.1) node {$B$};
\draw [fill=black] (-0.8452364011647344,-0.5343925768066846) circle (1.0pt);
\draw[color=black] (-0.98,-0.6) node {$C$};
\draw [fill=black] (0.34523640116473436,-0.9385157576252103) circle (1.0pt);
\draw[color=black] (0.36464180979875993,-1.05502879024742) node {$D$};
\draw [fill=black] (-0.4133155452748829,0.) circle (1.0pt);
\draw[color=black] (-0.423,-0.1959904769200991) node {$E$};
\draw [fill=black] (-2.42,0.) circle (1.0pt);
\draw[color=black] (-2.42,0.2) node {$G$};
\draw[color=black] (-0.423,2.7) node {$\ell$};
\draw [fill=black] (-0.5,1.4897) circle (1.0pt);
\draw[color=black] (-0.38,1.7) node {$T$};
\draw [fill=black] (1.9691, 1.4897) circle (1.0pt);
\draw[color=black] (1.9691, 1.6) node {$L$};
\draw[line width=1.pt,color=uququq,fill=uququq,fill opacity=0.10000000149011612] (-0.10756201875973295,0.79264858026931) -- (-0.15537024378307718,0.9350865615095769) -- (-0.2978082250233442,0.8872783364862328) -- (-0.25,0.7448403552459657) -- cycle; 
\draw[line width=0.2mm,color=black] (-2.8,0) -- (1.5,0);
\end{scriptsize}
\end{tikzpicture}
\caption{Lemma \ref{lemm.4.1} and Corollary \ref{cor.4.1}.}\label{fig.13}
\end{figure}

\begin{proof} (\textit{of the Theorem \ref{thm.4.1}}.)
From the last Lemma we see that $B$ and $D$ are symmetric with respect to the axis of the parabola as well as $A$ and $C$. Since the center of the circle also belongs to that axis, we see that $AB \cong CD$ and also $BC \cong AD$. Also, the intersection points $G\in AD \cap BC$ and $H \in AB \cap CD$ belong to the axis of $\mathcal{P}$. The diagonals $AC$ and $BD$ are orthogonal to the axis of $\mathcal{P}$, since $x_A=x_C$ and $x_B=x_D$. This proves (a)-(c).

It easily follows from the eq. \eqref{eq.3.0.6a} that the pedal curve of the parabola with the focus as the pedal point is the line $x=-p/2$. Thus, by Proposition \ref{prop.3.3}, the midpoint of the side $AB$ of the 4-Poncelet polygon $ABCD$ lies on the tangent to the parabola at its vertex. This also gives
\begin{equation*}
x_A+x_B=-p.
\end{equation*}
In other words, the sum of the $x$-coordinates of the endpoints any two sides is constant equal to $-p$. This proves (d), and ends the proof of the theorem.
\end{proof}

\begin{remark}\label{rem.4.2}
The previous consideration contains also the answer to the following question: {\it Given a directrix $\ell$ and the focus $F$ of a parabola $P$ and a point $A$, construct the tangents from $A$ to $P$.} 

\noindent
The analysis is given above.\\
\textit{Construction.} Denote $|AF|=R$. Construct the circle $C_1$ with the center at $F$ with radius $R$ and the circle $C_2$ with the center at $A$ with the same radius (see Figure \ref{fig.14}.) Denote by $T_1$ and $T_2$ the intersections of $C_2$ with $\ell$. Construct circles $C_3$ and $C_4$ with centers $T_1$ and $T_2$, respectively, and the same radius $R$. Denote the intersections of $C_3$ and $C_4$ with $C_1$, different than $A$, by $B$ and $D$, respectively. Then $AB$ and $AD$ are tangents to the parabola from $A$. The orthogonal projections of their points of contact with the parabola to the directrix $\ell$ are $T_1$ and $T_2$, respectively.


\noindent
{\it Discussion.} Denote by $-p$ the distance from the focus to $\ell$. If $A$ is such that $R<-p/2$ ($A$ is in the convex complement of the parabola), then the construction has no solutions. If $R=-p/2$ ($A$ is the vertex of the parabola), then the construction has one solution, the tangent to the parabola at the  vertex. If $R>-p/2$ ($A$ is in the nonconvex complement of the parabola), then the construction gives two solutions, the two tangents to the parabola through $A$.

\begin{figure}
    \centering
\definecolor{uuuuuu}{rgb}{0.26666666666666666,0.26666666666666666,0.26666666666666666}
\definecolor{wqwqwq}{rgb}{0.3764705882352941,0.3764705882352941,0.3764705882352941}
\definecolor{qqqqff}{rgb}{0.,0.,1.}
    \begin{subfigure}[b]{0.45\textwidth}
    \centering
\begin{tikzpicture}[scale=0.4]
\clip(-10,-8) rectangle (20,14);
\draw [samples=100,rotate around={-90.:(0.,0.)},xshift=0.cm,yshift=0.cm,line width=1.pt,domain=-16.0:16.0)] plot (\x,{(\x)^2/2/2.0});
\draw [line width=1.pt] (-1.,-11.226075342289871) -- (-1.,18.94754017552187);
\draw [line width=1.pt] (-2.9865734212459203,3.1260407590475383)-- (1.,0.);
\draw [line width=1.pt] (1.,0.) circle (5.066053540006354cm);
\draw [line width=1.pt] (-2.9865734212459203,3.1260407590475383) circle (5.066053540006354cm);
\draw [line width=1.pt] (-1.,7.786343808443632) circle (5.0660535400068145cm);
\draw [line width=1.pt] (-1.,-1.53426229034807) circle (5.0660535400068145cm);
\draw [line width=1.pt,color=qqqqff,domain=-2.9865734212459203:27.631274261607192] plot(\x,{(--23.25448746698115--1.534262290348341*\x)/5.973146842492257});
\draw [line width=1.pt,color=qqqqff,domain=-2.9865734212459203:27.631274261607192] plot(\x,{(-4.58218697757324-7.786343808443447*\x)/5.97314684249219});
\draw [line width=1.pt,dash pattern=on 3pt off 3pt] (15.156773181677831,7.786340136849361)-- (-1.,7.786343808443596);
\draw [line width=1.pt,dash pattern=on 3pt off 3pt] (0.5884910021584129,-1.53426334396467)-- (-1.,-1.5342622903485195);
\begin{scriptsize}
\draw[color=black] (8,5) node {$\mathcal{P}$};
\draw[color=black] (-1.7563920411312866,17.666101237902456) node {$eq2$};
\draw [fill=black] (1.,0.) circle (3.5pt);
\draw[color=black] (1.148202884139377,-0.410730708782039) node {$F$};
\draw [fill=black] (-2.9865734212459203,3.1260407590475383) circle (3.5pt);
\draw[color=black] (-3.089088536255473,3.8607324166159986) node {$A$};
\draw[color=black] (6.5,0) node {$C_1$};
\draw[color=black] (-7.873127236701272,5.77434789679432) node {$C_2$};
\draw [fill=black] (-1.,7.786343808443596) circle (3.5pt);
\draw[color=black] (-0.577468218521429,8.490998444547472) node {$T_1$};
\draw [fill=black] (-1.,-1.5342622903485195) circle (3.5pt);
\draw[color=black] (-0.5,-1.98) node {$T_2$};
\draw[color=black] (-4.4217850313796605,12.506173782421628) node {$C_3$};
\draw[color=black] (-6.5,-2) node {$C_4$};
\draw [fill=black] (2.9865734212463373,4.660303049395879) circle (3.5pt);
\draw[color=black] (2.7884447242922223,5.1) node {$B$};
\draw [fill=black] (2.98657342124627,-4.660303049395909) circle (3.5pt);
\draw[color=black] (2.92,-4) node {$D$};
\draw [fill=uuuuuu] (15.156773181677831,7.786340136849361) circle (3.5pt);
\draw[color=uuuuuu] (15.244031197952891,8.5) node {$L_1$};
\draw [fill=uuuuuu] (0.5884910021584129,-1.53426334396467) circle (3.5pt);
\draw[color=uuuuuu] (1.2,-1.5) node {$L_2$};
\draw[color=black] (-1.4,13.5) node {$\ell$};
\end{scriptsize}
\end{tikzpicture}
\caption{$A$ does not lie on the tangent at the vertex.}
\label{fig.13(A)}
\end{subfigure}

\begin{subfigure}[b]{0.45\textwidth}
\centering
\begin{tikzpicture}[scale=0.8]
\clip(-5,-4) rectangle (9,9);
\draw[line width=1.pt,color=wqwqwq,fill=wqwqwq,fill opacity=0.10000000149011612] (0.12452394593641374,2.1886901351589656) -- (0.43583381077744804,2.3132140810953796) -- (0.3113098648410343,2.6245239459364136) -- (0.,2.5) -- cycle; 
\draw [samples=100,rotate around={-90.:(0.,0.)},xshift=0.cm,yshift=0.cm,line width=1.pt,domain=-12.0:12.0)] plot (\x,{(\x)^2/2/2.0});
\draw [line width=1.pt] (-1.,-5.212516154917697) -- (-1.,8.743977555653235);
\draw [line width=1.pt] (0.,2.5)-- (1.,0.);
\draw [line width=1.pt] (1.,0.) circle (2.6925824035672523cm);
\draw [line width=1.pt] (0.,2.5) circle (2.6925824035672523cm);
\draw [line width=1.pt] (-1.,5.) circle (2.692582403567252cm);
\draw [line width=1.pt] (-1.,0.) circle (2.692582403567252cm);
\draw [line width=1.pt,color=qqqqff] (0.,-5.212516154917697) -- (0.,8.743977555653235);
\draw [line width=1.pt,color=qqqqff,domain=-5.304436432531913:16.8078338971269] plot(\x,{(--15.625--2.5*\x)/6.25});
\draw [line width=1.pt,dash pattern=on 3pt off 3pt] (6.25,5.)-- (-1.,5.);
\draw [line width=1.pt,dash pattern=on 3pt off 3pt] (0,0)-- (-1.,0);
\begin{scriptsize}
\draw[color=black] (4.400305513437,4.0259556161227446) node {$\mathcal{P}$};
\draw [fill=black] (1.,0.) circle (2.0pt);
\draw[color=black] (1.0652883300502887,-0.19418461119598615) node {$F$};
\draw [fill=black] (0.,2.5) circle (2.0pt);
\draw[color=black] (0.11694221155169801,2.793105662074576) node {$A$};
\draw[color=black] (3.5,1.5) node {$C_1$};
\draw[color=black] (-2.6,4.010149847481101) node {$C_2$};
\draw [fill=black] (-1.,5.) circle (2.0pt);
\draw[color=black] (-0.75,5.35) node {$T_1$};
\draw [fill=black] (-1.,0.) circle (2.0pt);
\draw[color=black] (-0.7,-0.4) node {$T_2$};
\draw[color=black] (-2.7122903753024317,7.42) node {$C_3$};
\draw[color=black] (-3.8,-1) node {$C_4$};
\draw [fill=black] (0.,-2.5) circle (2.0pt);
\draw[color=black] (0.2,-2.8) node {$D$};
\draw [fill=black] (6.25,5.) circle (2.0pt);
\draw[color=black] (6.2,5.3) node {$L_1$};
\draw [fill=black] (0.,0.) circle (2.0pt);
\draw[color=black] (0.3,0.1) node {$L_2$};
\draw[color=black] (-1.2,8.5) node {$\ell$};
\draw[color=wqwqwq] (0.8,2.4) node {$90^\circ$};
\end{scriptsize}
\end{tikzpicture}
\caption{$A$ lies on the tangent at the vertex.}
\label{fig.13(B)}
\end{subfigure}
    \caption{Geometric construction of the pair of tangents and the points of contact from a point $A$ to the parabola $\mathcal{P}$.}
    \label{fig.14}
\end{figure}
\end{remark}

The converse of the statement in Theorem \ref{thm.4.1}(d) is also true.
\begin{lemma}\label{lemm.4.3}
Given a parabola $\mathcal{P}$ with its focus $F$ assumed to be at the origin and the directrix to be the line $\ell: x=-p$ and given a circle centered at the focus of the parabola $\mathcal{P}$ with radius $R>-p/2$, consider two points $A$ and $B$ of the circle and denote their $x$ coordinates by $x_A$ and $x_B$ respectively. If
\begin{equation*}
    x_A+x_B=-p,
\end{equation*}
 then the chord $AB$ belongs to a tangent to the parabola.   
\end{lemma}

\begin{proof}
The proof follows from the previous Lemma and the facts that for a given point $A$ there are two tangents to the parabola containing $A$ and two points of the circle with $x$-coordinates being equal to $x=-p-x_A$.
\end{proof}

\begin{corollary}\label{cor.4.2}
Given a parabola $\mathcal{P}$  and a circle centered at the focus of the parabola. Then the tangent to the circle at the intersection of the circle with the tangent to the parabola at its vertex is tangent to the parabola. 
\end{corollary}

\begin{theorem}\label{thm.4.2}
Given a parabola and a circle centered at the focus of the parabola with the diameter bigger than the distance from the focus to the directrix of the parabola. Then, there exists a quadrilateral inscribed in the circle and circumscribed about the parabola. There are infinitely many such quadrilaterals. Moreover, there are infinitely many parabolas confocal with the initial one with the same property with the same circle.
\end{theorem}
\begin{proof} We assume that the axis of the parabola is selected to be the $x$-axis with the focus at the origin. Pick a point $A$ on the circle, in the nonconvex component of the complement of the parabola, and the point $C$ symmetric to $A$ with respect to the $x$-axis. Such a point $A$ exists due to the assumption that the diameter of the circle is bigger than the distance from the focus to the directrix of the parabola. Denote by $B$ and $C$ two points symmetric to each other with respect to the $x$-axis such that $x_B=x_C=-p-x_A$. (Alternatively, we can construct $B$ and $C$ using the construction from Remark \ref{rem.4.2}.) By the previous Lemma, the quadrilateral $ABCD$ is circumscribed about the parabola.

This proof gives also the poristic property, or a proof of the  Poncelet theorem in this case, since it provides a quadrilateral inscribed in the circle and circumscribed about the parabola with an arbitrary vertex $A$ on the circle in the nonconvex connected component of the complement of the parabola.

We see that the initial parabola can be replaced with any parabola confocal to it with the distance between the focus and the directrix smaller than the diameter of the circle.
\end{proof}

In the following theorem we prove that every Darboux butterfly is a 4-Poncelet polygon circumscribing some parabola. More precisely, we have the following:
\begin{theorem}\label{thm.4.3}
Given an isosceles trapezoid $ABDC$ with $AC \parallel BD$ and $AB \cong CD$, there exists a parabola $\mathcal{P}$ that inscribes the Darboux butterfly $ABCD$.
\end{theorem}
\begin{proof}
Every isosceles trapezoid is cyclic. Denote by $\mathcal{D}$ the circumcircle of $ABDC$ and by $F$ its circumcenter (see Figure \ref{fig.15}.) Construct the midline $\mathscr{m}$ of the trapezoid $ABDC$. Construct the parallel line $\ell$ at the same distance from $F$ to $\mathscr{m}$ such that $\mathscr{m}$ and $\mathscr{l}$ are at the same side of $F$. Then, by the Theorem \ref{thm.4.1}, the parabola $\mathcal{P}$, with the focus $F$ and directrix $\ell$ is 4-inscribed into the circle $\mathcal{C}$, inscribed in the butterfly $ABCD$ as a 4-Poncelet polygon.
\end{proof}

\begin{figure}
    \centering
\begin{tikzpicture}[scale=2.5]
\clip(-1.5,-1.5) rectangle (1.5,1.5);
\draw [line width=1.pt,dash pattern=on 3pt off 3pt] (0.,0.) circle (1.cm);
\draw [samples=150,rotate around={-90.:(-0.25,0.)},xshift=-0.25cm,yshift=0.cm,line width=1.pt,domain=-4.0:4.0)] plot (\x,{(\x)^2/2/0.5});
\draw [line width=1.pt] (-0.8139830015951613,0.5808886925342338)--(-0.8139830015951613,-0.5808886925342338);
\draw [line width=1.pt,color=blue] (-0.8139830015951613,0.5808886925342338)-- (0.3139830015951613,0.9494286043243552);
\draw [line width=1.pt,color=blue] (0.3139830015951613,0.9494286043243552)-- (-0.8139830015951612,-0.5808886925342337);
\draw [line width=1.pt] (0.3139830015951613,0.9494286043243552)-- (0.3139830015951613,-0.9494286043243552);
\draw [line width=1.pt,color=blue] (-0.8139830015951612,-0.5808886925342337)-- (0.3139830015951615,-0.9494286043243554);
\draw [line width=1.pt,color=blue] (0.3139830015951615,-0.9494286043243554)-- (-0.8139830015951613,0.5808886925342338);
\draw [line width=1.pt,dash pattern=on 2pt off 2pt,domain=-2.4254631983492443:5.5129940730589535] plot(\x,{(-0.--1.1279660031903227*\x)/-0.36853991179012136});
\draw [line width=1.pt,dash pattern=on 2pt off 2pt,domain=-2.4254631983492443:5.5129940730589535] plot(\x,{(-0.--1.1279660031903227*\x)/0.3685399117901217});
\begin{scriptsize}
\draw [fill=black] (0.,0.) circle (0.8pt);
\draw[color=black] (0.12,0.05585557841619577) node {$F$};
\draw[color=black] (-0.6496691111245437,0.9371756068906757) node {$\mathcal{D}$};
\draw[color=black] (1.2905503545468882,1.127909045888884) node {$\mathcal{P}$};
\draw [fill=black] (-0.8139830015951613,0.5808886925342338) circle (0.8pt);
\draw[color=black] (-0.89,0.6872490316516441) node {$A$};
\draw [fill=black] (0.3139830015951613,0.9494286043243552) circle (0.8pt);
\draw[color=black] (0.42,1.048984864234453) node {$B$};
\draw [fill=black] (-0.8139830015951612,-0.5808886925342337) circle (0.8pt);
\draw[color=black] (-0.9587888226043989,-0.6) node {$C$};
\draw [fill=black] (0.3139830015951615,-0.9494286043243554) circle (0.8pt);
\draw[color=black] (0.43,-1.0) node {$D$};
\draw [fill=black] (-0.25,0.7651586484292945) circle (0.8pt);
\draw [fill=black] (-0.25,-0.7651586484292945) circle (0.8pt);
\draw[color=black] (-0.7877864290197982,2.5880064064958583) node {$r$};
\draw[color=black] (0.9288145219640789,2.5880064064958583) node {$s$};
\end{scriptsize}
\end{tikzpicture}
    \caption{Theorem \ref{thm.4.3}.}
    \label{fig.15}
\end{figure}

\begin{corollary}\label{cor.4.3}
    Given a circle $\mathcal{C}$ centered at a point $F$, and a point $E$, different from $F$, not on $\mathcal{C}$. For every point $A$ on the circle  $\mathcal{C}$ outside $EF$ line, there exists a unique parabola $\mathcal{P}$ with the focus $F$ such that a quadrilateral $ABCD$ can be inscribed in the circle and circumscribed about the parabola, such that $E$ is the intersection of the lines $AD$ and $BC$, in the case of $E$  inside $\mathcal{C}$, or $E$ is the intersection of the lines $AB$ and $CD$, in the case of $E$ outside $\mathcal{C}$.
\end{corollary}
\begin{proof} 
    Let the point $E$ be inside $\mathcal{C}$. Denote by $G$ the inverse of $E$ with respect to the circle $\mathcal{C}$. Choose a point $A$ on $\mathcal{C}$ outside the line $EF$. Denote by $B$ the other intersection of $GA$ with $\mathcal{C}$. Denote by $D$ and $C$ the intersections of $AE$ and $BE$ with $\mathcal{C}$, respectively. Then $G,C,D$ are collinear. Indeed, consider the axial reflection with respect to the line $EF$. It maps $A$ to $C$ and $B$ to $D$, while keeps $G$ fixed. Thus, the line $GAB$ maps to the line $GCD$. Thus, $ABCD$ has two pairs of congruent opposite sides. So, it is a Darboux butterfly. Now the existence of a parabola follows from the Theorem \ref{thm.4.2}.

    Similarly, if $E$ lies outside $\mathcal{C}$, then by finding $G$, inside $\mathcal{C}$, the rest of the proof can be completed by relabeling the points $E$ and $G$ in the previous proof.
\end{proof}

\begin{theorem}\label{thm.4.4}
For an isosceles trapezoid, the intersection of its diagonals and the intersection of the lines of its legs (congruent sides) are symmetric with respect to the circumcircle of the trapezoid. 

In other words, for an antiparallelogram (Darboux butterfly), the points of intersections of the pairs of opposite sides are inverse with respect to the circumcircle of the antiparallelogram.
\end{theorem}

\begin{proof}
Every isosceles trapezoid is cyclic. Without loss of generality, let us suppose $AC||BD$ and $AB \cong  CD$ (see Figure \ref{fig.12}.) Denote the circumcircle by $\mathcal{D}$ and its center by $F$. Now, according to Theorem \ref{thm.4.2}, there exists a parabola $\mathcal P$ with the focus $F$, such that the Darboux butterfly $ABCD$ is a Poncelet quadrilateral associated with the $4$-Poncelet pair $(\mathcal{D}, \mathcal{P})$.

Let $F=(0,0)$, $A\in \mathcal{D}(F)$, and the circle and the parabola have the following equations:
\begin{eqnarray*}
    \mathcal{D}(F): x^2+y^2=1, \quad \mathcal{P}(p): y^2=2px+p^2,
\end{eqnarray*}
respectively.

By Theorem \ref{thm.4.1}, the points $G \in AD \cap BC$ and $H \in AB \cap CD$ lie on the $x$-axis. Thus, by substituting $y=0$ into equation \eqref{eq.2.1.2} for a pair of tangents to parabola the $\mathcal{P}$ from  the point $A\in \mathcal{D}(F)$, one gets
\begin{equation*}
p x^2 + 2(y_A^2 - p x_A) x+p=0.
\end{equation*}
Since $x_G$ and $x_H$ are the solutions of the above quadratic equation in $x$, we have
\begin{equation*}
    x_G x_H = 1.
\end{equation*}
This concludes the proof.
\end{proof}

\begin{proposition}\label{prop.4.1}
Let $\mathcal{F}$ be a given confocal family of parabolas with the focus $F$. Let $\mathcal{D}$ be a circle centered at $F$. Then every chord of $\mathcal{D}$ that does not contain the focus and is not parallel to the axis of the confocal family $\mathcal{F}$ is tangent to a unique parabola $\mathcal{P}\in \mathcal{F}$.
\end{proposition}

\begin{proof} We assume that the axis $\mathscr{a}$ of the parabolas from the confocal pencil is the $x$-axis.
    Let the chord be $\overline{AB}$ and $m_{AB}$ denotes the slope of the line $AB$. By assumption, $m_{AB} \neq 0$. Also, $y_A\neq m_{AB}x_A$ because by assumption $F\notin AB$. Thus, $AB$ is tangent to a unique nondegenerate parabola $\mathcal{P}(p)$ if and only if $p\ne 0$ is given by the eq. \eqref{eq.2.1.2}:
    \begin{equation*}
        p=\frac{2m_{AB}(y_A-m_{AB}x_A)}{m_{AB}^2+1}.
    \end{equation*}
\end{proof}

\subsection{The focus of the parabola differs from the center of the circle}\label{sec.4.2}

\begin{figure}
    \centering
\definecolor{ffqqqq}{rgb}{1.,0.,0.}
\definecolor{uuuuuu}{rgb}{0.26666666666666666,0.26666666666666666,0.26666666666666666}
\definecolor{uququq}{rgb}{0.25098039215686274,0.25098039215686274,0.25098039215686274}
\begin{tikzpicture}[scale=1.6]
\clip(-3.5,-2) rectangle (2.5,3.5);
\draw [samples=100,rotate around={-90.:(-0.8,0.)},xshift=-0.8cm,yshift=0.cm,line width=1.pt,domain=-6.4:6.4)] plot (\x,{(\x)^2/2/1.6});
\draw[line width=0.8pt,color=uququq,fill=uququq,fill opacity=0.10000000149011612] (0.07975744700078513,0.15951489400157032) -- (-0.07975744700078517,0.23927234100235542) -- (-0.1595148940015703,0.07975744700078517) -- (0.,0.) -- cycle;
\draw [line width=1.pt] (-2.,1.) circle (1.cm);
\draw [line width=1.pt,color=ffqqqq,domain=-5.331837875586127:4.868512763996262] plot(\x,{(-0.-4.070391530819715*\x)/-2.035195765409858});
\draw [line width=1.pt,dash pattern=on 2pt off 2pt,domain=-4.265191537347405:3.133120787210222] plot(\x,{(-0.-1.*\x)/2.});
\draw [line width=1.pt,dash pattern=on 3pt off 3pt] (-2.935706724643765,1.3527788619750858)-- (-0.5323088442077114,-1.0646176884154228);
\draw [line width=1.pt,dash pattern=on 3pt off 3pt] (-2.935706724643765,1.3527788619750858)-- (1.5028869212021467,3.005773842404292);
\draw [line width=1.pt,dash pattern=on 3pt off 3pt] (-1.5229863524694462,1.8788958869340537)-- (-0.5323088442077114,-1.0646176884154228);
\draw [line width=1.pt,dash pattern=on 3pt off 3pt] (-1.6526612701470038,0.06226026705481477)-- (1.5028869212021467,3.005773842404292);
\draw [line width=1.pt] (-2.935706724643765,1.3527788619750858)-- (-1.5229863524694462,1.8788958869340537);
\draw [line width=1.pt] (-1.5229863524694462,1.8788958869340537)-- (-1.0886456527397848,0.5883772920137824);
\draw [line width=1.pt] (-1.0886456527397848,0.5883772920137824)-- (-1.6526612701470038,0.06226026705481477);
\draw [line width=1.pt] (-1.6526612701470038,0.06226026705481477)-- (-2.935706724643765,1.3527788619750858);
\begin{scriptsize}
\draw [fill=black] (-2.,1.) circle (1.0pt);
\draw[color=black] (-2.5,2) node {$\mathcal{D}$};
\draw[color=black] (-1.9280372287915704,1.1871647015055784) node {$E$};
\draw [fill=black] (-2.935706724643765,1.3527788619750858) circle (1.0pt);
\draw[color=black] (-3.1,1.4) node {$A$};
\draw [fill=uuuuuu] (-1.5229863524694462,1.8788958869340537) circle (1.0pt);
\draw[color=uuuuuu] (-1.47,2.0602174648067724) node {$B$};
\draw [fill=uuuuuu] (-1.0886456527397848,0.5883772920137824) circle (1.0pt);
\draw[color=uuuuuu] (-0.94,0.55) node {$C$};
\draw [fill=uuuuuu] (-1.6526612701470038,0.06226026705481477) circle (1.0pt);
\draw[color=uuuuuu] (-1.66,-0.07) node {$D$};
\draw [fill=uuuuuu] (1.5028869212021467,3.005773842404292) circle (1.0pt);
\draw[color=uuuuuu] (1.6304816292209203,3.05) node {$I$};
\draw [fill=uuuuuu] (-0.5323088442077114,-1.0646176884154228) circle (1.0pt);
\draw[color=uuuuuu] (-0.45,-1.25) node {$J$};
\draw[color=uuuuuu] (-0.75,-1.7) node {$\mathscr{k}$};
\draw[color=uuuuuu] (1.8,2.7) node {$\mathcal{P}$};
\draw [fill=black] (0.,0.) circle (1.0pt);
\draw[color=black] (0.06,-0.2) node {$F$};
\end{scriptsize}
\end{tikzpicture}
\begin{tikzpicture}[scale=1.7]
\clip(-3.5,-2) rectangle (2,3.5);
\draw [samples=100,rotate around={-90.:(-0.8,0.)},xshift=-0.8cm,yshift=0.cm,line width=1.pt,domain=-6.4:6.4)] plot (\x,{(\x)^2/2/1.6});
\draw [line width=1.pt] (-2.,1.) circle (1.cm);
\draw [line width=1.pt,color=ffqqqq,domain=-4.265191537347405:3.133120787210222] plot(\x,{(-0.-3.9545917674187514*\x)/-1.977295883709375});
\draw [line width=1.pt,dash pattern=on 2pt off 2pt,domain=-4.265191537347405:3.133120787210222] plot(\x,{(-0.-1.*\x)/2.});
\draw [line width=1.pt,dash pattern=on 3pt off 3pt](-2.9983525516248184,0.9426224550523913)-- (-0.5674297645005877,-1.1348595290011758);
\draw [line width=1.pt,dash pattern=on 3pt off 3pt] (-2.9983525516248184,0.9426224550523913)-- (1.4098661192087873,2.8197322384175756);
\draw [line width=1.pt,dash pattern=on 3pt off 3pt] (-1.2667600639771412,1.6799699965595494)-- (-0.5674297645005877,-1.1348595290011758);
\draw [line width=1.pt,dash pattern=on 3pt off 3pt] (-1.9010990944412305,0.004902712856851732)-- (1.4098661192087873,2.8197322384175756);
\draw [line width=1.pt] (-2.9983525516248184,0.9426224550523913)-- (-1.2667600639771412,1.6799699965595494);
\draw [line width=1.pt] (-1.2667600639771412,1.6799699965595494)-- (-1.0337882899568118,0.7422502543640087);
\draw [line width=1.pt] (-1.0337882899568118,0.7422502543640087)-- (-1.9010990944412305,0.004902712856851732);
\draw [line width=1.pt] (-1.9010990944412305,0.004902712856851732)-- (-2.9983525516248184,0.9426224550523913);
\draw [line width=1.pt] (-2.9983525516248184,0.9426224550523913)-- (-1.0337882899568118,0.7422502543640087);
\draw [line width=1.pt] (-1.2667600639771412,1.6799699965595494)-- (-1.9010990944412305,0.004902712856851732);
\draw [line width=0.7pt] (-1.6,-2.4138134915500586) -- (-1.6,3.942009369092581);
\draw [line width=1.pt] (-2.13255630780098,1.3112962258059704)-- (-1.467443692199021,0.3735764836104302);
\draw [line width=1.pt] (-1.1502741769669766,1.211110125461779)-- (-2.4497258230330243,0.4737625839546215);
\draw [line width=1.pt,dash pattern=on 2pt off 2pt] (-1.8,-0.5) -- (-1.8,2.5);
\draw [line width=1.pt,color=blue] (-1.6,0.8)--(1.4098661192087873,2.8197322384175756);
\draw [line width=1.pt,color=blue] (-1.6,0.8)--(-0.5674297645005877,-1.1348595290011758);
\begin{scriptsize}
\draw [fill=black] (-2.,1.) circle (1.0pt);
\draw[color=black] (-2.15,0.95) node {$E$};
\draw[color=black] (-2.558535398841497,2.004155970444263) node {$\mathcal{D}$};
\draw [fill=black] (-2.9983525516248184,0.9426224550523913) circle (1.0pt);
\draw[color=black] (-3.1722590348559376,0.9616665065293326) node {$A$};
\draw [fill=black] (-1.2667600639771412,1.6799699965595494) circle (1.0pt);
\draw[color=black] (-1.2470163958517368,1.8276053354264121) node {$B$};
\draw [fill=black] (-1.0337882899568118,0.7422502543640087) circle (1.0pt);
\draw[color=black] (-0.9275438182003846,0.7430800060310407) node {$C$};
\draw [fill=black] (-1.9010990944412305,0.004902712856851732) circle (1.0pt);
\draw[color=black] (-2.020476320691852,-0.1144516497699503) node {$D$};
\draw [fill=black] (-2.13255630780098,1.3112962258059704) circle (1.0pt);
\draw [fill=black] (-1.1502741769669766,1.211110125461779) circle (1.0pt);
\draw [fill=black] (-1.467443692199021,0.3735764836104302) circle (1.0pt);
\draw [fill=black] (-2.4497258230330243,0.4737625839546215) circle (1.0pt);
\draw [fill=black] (1.4098661192087873,2.8197322384175756) circle (1.0pt);
\draw[color=black] (1.55,2.9205378379178715) node {$I$};
\draw [fill=black] (-0.5674297645005877,-1.1348595290011758) circle (1.0pt);
\draw[color=black] (-0.5,-1.2494200177418502) node {$J$};
\draw[color=ffqqqq] (-1.0788729339299725,-2.2750951354646043) node {$l$};
\draw [fill=black] (0.,0.) circle (1.0pt);
\draw[color=black] (0.01405956856149517,-0.13126599596212662) node {$F$};
\draw [fill=black] (-1.6,0.8) circle (1.0pt);
\draw[color=black] (-1.45,0.92803781414498) node {$L$};
\draw [fill=black] (-1.8,0.8424363547082003) circle (1.0pt);
\draw[color=black] (-1.7,1) node {$G$};
\draw[color=black] (-1.45,2.5) node {$\ell$};
\draw[color=black] (-1.9,2.5) node {$\ell'$};
\draw[color=black] (1.8,2.7) node {$\mathcal{P}$};
\end{scriptsize}
\end{tikzpicture}
    \caption{Theorem \ref{thm.4.5}.}
    \label{fig.16}
\end{figure}

In this section we will study $4$-Poncelet pairs of circles and parabolas when the center of a circle is different from the focus of a parabola. We derive interesting properties of associated Poncelet quadrilaterals.

Let us recall that for a cyclic quadrilateral, the four lines, called {\it  maltitudes}, each passing  through a midpoint of a side of the quadrilateral and being orthogonal to the opposite side, intersect at one point, called {\it the anticenter of the cyclic quadrilateral}, see, for example, \cite{Honsberger1995}.

\begin{theorem}\label{thm.4.5}
Given a circle $\mathcal {D}$ with the center $E$ and a point $F\ne E$.  Let $XY$ be the polar of $F$ with respect to $\mathcal {D}$ and $L=EF\cap XY$. Let $\ell$ be a line containing $L$. Denote by $\mathcal{P}$ the parabola with the focus $F$ and the directrix $\ell$ (Figures \ref{fig.16}-\ref{fig.17}.) 
\begin{itemize}
    
\item[(a)] For any point $A$ of the circle $\mathcal {D}$ in the nonconvex connected component of the complement of $\mathcal P$, denote by $B$ and $D$ the intersections with the circle of the tangents through $A$ to the parabola. Then $L\in BD$.

\item[(b)] For any point $A$ of the circle $\mathcal {D}$ in the nonconvex connected component of the complement of $\mathcal P$, there exist a quadrilateral $ABCD$ inscribed in the circle and circumscribed about the parabola. Thus, $(\mathcal D, \mathcal P)$ form a $4$-Poncelet pair.
\item[(c)] For every Poncelet quadrilateral $ABCD$, associated with $(\mathcal D, \mathcal P)$, the intersection of the diagonals $AC\cap BD=L$, thus does not depend on the quadrilateral.
\item[(d)] The points of intersections of the opposite sides of any Poncelet quadrilateral, associated
with $(\mathcal{D},\mathcal{P})$, lie on the line $\mathscr{k}$ that contains $F$ and is orthogonal to $EF$.
\item[(e)] The anticenters lie on the directrix $\ell$.
\item[(f)] The centroids lie on a line $\ell'$, parallel to $\ell$. If $\ell$ is set to be parallel to $y$-axis and $F$ at the origin, then, the $x$-coordinate of the centroid $G$ is:
$$
 x_G=\frac{x_E-p}{2}.
$$
\end{itemize}
\end{theorem}

\begin{proof} Let the coordinates of $E$ and $F$ be $(x_E, y_E)$ and $(0,0)$, respectively. The equations of the lines $EF$ and $XY$ are: 
    \begin{eqnarray*}
        EF&:& y=\frac{y_E}{x_E}x,\\
        XY&:& -x_E(x-x_E)-y_E(y-y_E)=1.
    \end{eqnarray*}
Then, the coordinates of their intersection are:
\begin{equation*}
    L=(x_L,y_L)=\left(\frac{x_E(x_E^2+y_E^2-1)}{x_E^2+y_E^2},\frac{y_E(x_E^2+y_E^2-1)}{x_E^2+y_E^2}\right).
\end{equation*}

If $\ell: x=-p$ is the line containing $L$, then  we have
\begin{eqnarray}\label{eq.4.2.1}
    p=-x_L=-\frac{x_E(x_E^2+y_E^2-1)}{x_E^2+y_E^2}.
\end{eqnarray}

Take an arbitrary point $A\in \mathcal{D}(E)$. Let the pair of tangents from $A$ to the parabola $\mathcal{P}(p)$ intersect $\mathcal{D}(E)$ at $B,D$.\\ \noindent

\begin{figure}
    \centering
    \begin{tikzpicture}[scale=1.6]
\clip(-2.5,-1.5) rectangle (1.5,2.5);
\draw [line width=1.pt] (-1.,1.) circle (1.cm);
\draw [samples=100,rotate around={-90.:(-0.25,0.)},xshift=-0.25cm,yshift=0.cm,line width=1.pt,domain=-4.0:4.0)] plot (\x,{(\x)^2/2/0.5});
\draw [line width=1.pt,domain=-2.500549106251379:5.532528525682692] plot(\x,{(-1.-1.*\x)/-1.});
\draw [line width=1.pt] (0.,0.)-- (-1.,1.);
\draw [line width=1.pt] (-0.5,-2.177422743934339) -- (-0.5,3.561418315125348);
\draw [line width=1.pt] (-0.5,-2.177422743934339) -- (-0.5,3.561418315125348);
\draw [line width=1.pt] (0,0) -- (0,1);
\draw [line width=1.pt] (0,0) -- (-1,0);
\begin{scriptsize}
\draw [fill=black] (0.,0.) circle (1.0pt);
\draw[color=black] (0.08,0.10706494831875268) node {$F$};
\draw [fill=black] (-1.,1.) circle (1.0pt);
\draw[color=black] (-0.9537267143417442,1.11) node {$E$};
\draw [fill=black] (-0.5,0.5) circle (1.0pt);
\draw[color=black] (-0.35,0.52) node {$L$};
\draw[color=black] (-1.4476699991532243,2.0113462963419555) node {$\mathcal{D}$};
\draw[color=black] (0.8075710512360336,0.85) node {$\mathcal{P}(p)$};
\draw[color=black] (2.406387473125825,3.460679882038796) node {$n$};
\draw [fill=black] (0.,1.) circle (1.0pt);
\draw[color=black] (0.12,0.95) node {$X$};
\draw [fill=black] (-1.,0.) circle (1.0pt);
\draw[color=black] (-0.9992214905743806,-0.15) node {$Y$};
\draw[color=black] (-0.4,2.0113462963419555) node {$\ell$};
\draw[color=black] (0.8075710512360336,2) node {$\mathscr{p}$};
\end{scriptsize}
\end{tikzpicture} \qquad
\begin{tikzpicture}[scale=1.6]
\clip(-2.5,-1.5) rectangle (1.5,2.5);
\draw [line width=1.pt] (-1.,1.) circle (1.cm);
\draw [samples=100,rotate around={-90.:(-0.25,0.)},xshift=-0.25cm,yshift=0.cm,line width=1.pt,domain=-4.0:4.0)] plot (\x,{(\x)^2/2/0.5});
\draw [line width=1.pt,domain=-2.500549106251379:4.401658373614304] plot(\x,{(-1.048631498298016-1.7749542789308035*\x)/2.072348707965907});
\draw [line width=1.pt,domain=-2.500549106251379:4.401658373614304] plot(\x,{(--16.05969203021823--1.7749542789308035*\x)/10.529502061213165});
\draw [line width=1.pt] (-1.9815551811998597,1.1911790424175992)-- (-0.22283903724729073,0.37069813445908006);
\draw [line width=1.pt,dash pattern=on 3pt off 3pt] (-0.13535975350151352,1.5023915247444333)-- (-0.22283903724729073,0.37069813445908006);
\draw [line width=1.pt,dash pattern=on 3pt off 3pt] (-0.22283903724729073,0.37069813445908006)-- (-0.6602460280513358,0.0594856521322461);
\begin{scriptsize}
\draw [fill=black] (0.,0.) circle (1.0pt);
\draw[color=black] (0.04715836277625536,0.10706494831875268) node {$F$};
\draw [fill=black] (-1.,1.) circle (1.0pt);
\draw[color=black] (-0.9537267143417439,1.11) node {$E$};
\draw [fill=black] (-0.5,0.5) circle (1.0pt);
\draw[color=black] (-0.45328417578274427,0.6075074868777514) node {$L$};
\draw[color=black] (-1.447669999153224,2.0113462963419555) node {$\mathcal{D}$};
\draw [fill=black] (-1.9815551811998597,1.1911790424175992) circle (1.0pt);
\draw[color=black] (-2.0261035826824574,1.3419231603474508) node {$A$};
\draw [fill=black] (-0.22283903724729073,0.37069813445908006) circle (1.0pt);
\draw[color=black] (-0.3,0.17) node {$C$};
\draw[color=black] (0.807571051236034,0.85) node {$\mathcal{P}(p)$};
\draw [fill=black] (-0.13535975350151352,1.5023915247444333) circle (1.0pt);
\draw[color=black] (-0.08932596592165362,1.6083925639957488) node {$B$};
\draw [fill=black] (-0.6602460280513358,0.0594856521322461) circle (1.0pt);
\draw[color=black] (-0.6872573106934454,-0.05) node {$D$};
\end{scriptsize}
\end{tikzpicture}
    \caption{$L=XY \cap EF$ and $L\in BD$.}
    \label{fig.17}
\end{figure}

We claim $L\in BD$. Indeed, 
\begin{eqnarray*}
    m_{BD}(x_L-x_B)+y_B&=&y_L-\frac{(x_A-x_E)((x_A-x_E)^2+(y_A-y_E)^2-1)}{x_Ay_E-x_Ey_A}\\
    &=& y_L,
\end{eqnarray*}
where we calculated $m_{BD}$ by using the eq. \eqref{eq.2.1.6}. The last equality follows from the assumption that $A\in \mathcal{D}(E)$.
This proves (a).

Denote by $C$ the other intersection of $AL$ and $\mathcal{D}(E)$. Let the pair of tangents from $C$ to the parabola $\mathcal{P}(p)$ intersect $\mathcal{D}(E)$ at $B'$ and $D'$. Then $L\in B'D'$.

Moreover, by using eq. \eqref{eq.2.1.6} we obtain $m_{BD}=m_{B'D'}$.
This proves that $B=B'$ and $D=D'$. Therefore, $ABCD$ is a quadrilateral inscribed in the circle and circumscribed about the parabola. Thus, $(\mathcal{D}(E),\mathcal{P}(p))$, where $p$ is given in eq. \eqref{eq.4.2.1}, form a 4-Poncelet pair.
This proves (b).

Now, (c) follows directly from (a) and (b).

We want to prove (d).
Take an arbitrary Poncelet quadrilateral $ABCD$ inscribed in $\mathcal{D}$ with the center $E$ and circumscribed about a parabola $\mathcal{P}$ with the focus $F$. Denote $L=AC\cap BD$, $I=AB\cap CD$, $J=AD\cap BC$. 

The $\triangle IJL$ is self-polar with respect to $\mathcal D$. Thus, $E$ is the orthocenter of $\triangle IJL$. From there we see that $EL \perp IJ$. 
By construction, $F\in EL$. Thus, $EF \perp IJ$.

Since $L$ is the pole of $IJ$ with respect to $\mathcal D$, the pole of every line passing through $L$ lies on $IJ$. Thus, $F$ belongs to $IJ$, since its polar line  with respect to $\mathcal D$ contains $L$, by construction. This concludes the proof of (d).

To prove that the anticenter lies on the directrix it is sufficient to prove that the line defined by mid-points of the diagonals of the quadrilateral $ABCD$ is orthogonal to the directrix $\ell$. This line is known as {\it the Newton-Gauss line}. According to a theorem of Newton (Proposition XXVII. Problem XIX. in \cite{Newton1687}, p. 82-84) if a central conic is inscribed in a quadrilateral, then its Newton-Gauss line contains the center of the conic. Now, we recast this for parabolas, and we get that the Newton-Gauss line {\it is parallel to the axis of the parabola}. Thus, it is indeed orthogonal to $\ell$. This proves (e). In particular, we proved that $y_A+y_C=y_B+y_D$.
   
Now, we use the fact that the anticenter $T$ is the orthocenter of the triangle $M_{AC}M_{BD}L$, where $M_{AC}$ and $M_{BD}$ are the midpoints of the diagonals $AC$ and $BD$, respectively.  Since the directrix $\ell$ passes through $L$ and is orthogonal to $M_{AC}M_{BD}$, the orthocenter lies on $\ell$.

Since the centroid $G$ is the midpoint of the segment $ET$, we get 
    \begin{equation}\label{eq.4.2.2}
        x_A+x_B+x_C+x_D=2(x_E-p).
    \end{equation}
This proves (f) and the proof is complete. 
\end{proof}

\begin{remark}\label{rem.4.3}
In the notation of the above theorem, for any Poncelet quadrilateral $ABCD$, the centroid $G$ and $F$ belong to the nine point circle of $\triangle IJL$, where  $L=AC\cap BD$, $I=AB\cap CD$, $J=AD\cap BC$. 
\end{remark}

\begin{remark}\label{rem.4.4}
There is a unifying observation putting together cases $E=F$ from the previous subsection  and $E\ne F$ from the current subsection: in Theorem \ref{thm.4.1} one can take $L$ to be the point at infinity corresponding to the directrices of the parabolas from the confocal family and the diagonals of a Darboux butterfly giving a Poncelet quadrilateral and relate that to Theorem \ref{thm.4.5}.
\end{remark}

\begin{remark}\label{rem.4.5}
Given a circle $\mathcal {D}$ with the center $E$ and a point $F\ne E$.  Let $XY$ be the polar of $F$ with respect to $\mathcal {D}$ and $L=EF\cap XY$. Let $\ell$ be a line containing $L$. Denote by $\mathcal{P}$ the parabola with the focus $F$ and the directrix $\ell$. 
The degenerate Poncelet quadrilaterals associated with $(\mathcal{D}, \mathcal{P})$ are formed by pairs of points of tangency at $\mathcal D$ of the common tangents to the circle and the parabola, see Proposition \ref{prop.2.2}.

If $x_{A_i}$, $i=1,2,3,4$ are the zeros of the eq. \eqref{eq.2.2.5}, then one can observe that
\begin{eqnarray}
    x_{A_1}+x_{A_2}+x_{A_3}+x_{A_4} 
    &=& 2(x_E-p)\label{eq.4.2.3}.
\end{eqnarray}

The equation \eqref{eq.4.2.3}, multiplied by two, as for a degenerate quadrilateral each vertex counts twice, can  be seen now as an instance of the equation \eqref{eq.4.2.2}, applied twice, once for each of the two degenerate quadrilaterals.
\end{remark}

\begin{theorem}\label{thm.4.6}
Let a circle  $\mathcal{D}$ with the center at a point $E$ and a parabola $\mathcal{P}$ with the focus $F$ and the directrix $\ell$ form a $4$-Poncelet pair and $E\neq F$. 
Let $f$ be the polar of $F$ with respect to $\mathcal{D}$ and denote by $L=EF\cap f$. Then $L\in\ell$.
\end{theorem}
\begin{proof} Take a nondegenerate Poncelet quadrilateral $ABCD$ associated with $(\mathcal D, \mathcal P)$, and denote $\hat L=AC\cap BD$, $I=AB\cap CD$, $J=AD\cap BC$.
Denote $\hat LE\cap IJ=\hat F$. Consider the pencil of parabolas $\mathcal{G}$ with the focus $\hat F$ and a directrix containing $\hat L$.
Conics that are tangent to the four lines $AB$, $BC$, $AD$, $CD$ form a pencil. Thus, this pencil coincides with $\mathcal{G}$. Thus, $\mathcal P$ has the focus $\hat F$ and its directrix contains $\hat L$.  Thus, $\hat F=F$. Since $\hat L$ is the pole of $IJ$, the polar of $F$ contains $\hat L$. Thus, $\hat L=L$. This ends the proof.
\end{proof}

\begin{theorem}\label{thm.4.7}
Given a circle $\mathcal {D}$ with the center $E$ and a confocal family of parabolas $\mathcal F$ with the focus $F\ne E$ and the axis of symmetry $\mathscr{a}$. Then there exists a unique parabola $\mathcal {P}\in \mathcal {F}$ such that $(\mathcal {D}, \mathcal {P})$ form a $4$-Poncelet pair.
\end{theorem}

\begin{proof} We construct the polar $XY$ of $F$ with respect to the circle $\mathcal D$, Denote by $L$ the intersection of $XY$ and $EF$. Then the directrix of a parabola $\mathcal P \in \mathcal F$ is orthogonal to the axis $a$ and contains $L$. Thus, it is determined uniquely. See Figure \ref{fig.17}.
\end{proof}

\begin{corollary}\label{cor.4.4}
Given a non-degenerate cyclic quadrilateral $ABCD$ with $L=AC\cap BD$, $I=AB\cap CD$, $J=AD\cap BC$ and $E$ the circumcenter, a parabola $\mathcal P$ is inscribed in $ABCD$ if and only if its focal point is $F=IJ\cap EL$ and the directrix is a line $\ell$ that contains $L$.
\end{corollary}
\begin{proof}
There are two cases to be considered. One is when $AC$ and $BD$ are parallel. Then they determine the point at infinity. Since every cyclic trapezoid is isosceles, the proof follows from the results of the previous subsection, dealing with the case $E=F$ and a confocal pencil of parabolas with the focus $F$ and directrices parallel to $AC$. 

The case when $AC$ and $BD$ are not parallel follows from Theorem \ref{thm.4.5} in one direction and from Theorem \ref{thm.4.6} in another.
\end{proof}

\section{Isoperiodic families of parabolas with a circle for $n=3$ and $n=4$}\label{sec.5}

The concept of isoperiodic families was introduced and applied in \cite{Dragovic-Radnovic24} to investigate $n$-isoperiodicity for a pair of confocal family of ellipses and hyperbolas with a circle. It was proved there that such families exist for $n=4$ and $n=6$ only and all such families were described. Then in \cite{Dragovic-Murad2025a} the isoperiodic families were studied for the case of a circle and confocal family of parabolas. It was proved that such families exist for $n=3$ and $n=4$ only and all such families were described. Both \cite{Dragovic-Radnovic24} and \cite{Dragovic-Murad2025a} used the Cayley condition.

Here we adopt an equivalent definition of isoperiodicity.

\begin{definition}\label{def.5.1}
A family $\mathcal{F}$ of nondegenerate conics is said to be an \textit{$n$-isoperiodic with $\mathcal{D}$} if there exist a nondegenerate conic $\mathcal{D}$ and a natural number $n \geq 3$ such that $(\mathcal{D},\mathcal{C})$ is an $n$-Poncelet pair for infinitely many  conics $\mathcal{C}\in \mathcal{F}$.
\end{definition}

From the previous results, we have the following theorem:
\clearpage
\begin{theorem}\label{thm.5.1}
Given  a circle $\mathcal {D}(E)$ with the center $E$, then \begin{itemize}
\item[(a)] For $n=3$, a confocal family of parabolas $\mathcal F$ with the focus $F$ and an axis of symmetry $\mathscr{a}$ is $3$-isoperiodic with the circle $\mathcal {D}(E)$ if and only if $F\in \mathcal D(E)$.

\item [(b)] For $n=4$, a confocal family of parabolas $\mathcal F$ with the focus $F$ and an axis of symmetry $\mathscr{a}$ is $4$-isoperiodic with the circle $\mathcal {D}(E)$ if and only if $F=E$.

\item[(c)]  For $n=4$ and $F\ne E$, a family of parabolas $\mathcal{G}$ with the focus $F$ is $4$-isoperiodic with the circle $\mathcal {D}(E)$ if and only if the directrices of the parabolas from $\mathcal{G}$ pass through the point of the intersection of the polar of $F$ with respect to  $\mathcal {D}(E)$  with the line $EF$.
\end{itemize}
\end{theorem}

\begin{figure}[htbp]
  \centering
  \definecolor{sqsqsq}{rgb}{0.12549019607843137,0.12549019607843137,0.12549019607843137}
  \begin{subfigure}[b]{0.45\textwidth}
    \centering
\begin{tikzpicture}[scale=1.5]
\clip(-2.5,-2) rectangle (2,2.5);
\draw [line width=1.pt,color=sqsqsq] (-0.8701192695765199,0.4928412084150431) circle (1.cm);
\draw [color=sqsqsq,samples=100,rotate around={-90.:(-0.5,0.)},xshift=-0.5cm,yshift=0.cm,line width=1.pt,domain=-6.0:6.0)] plot (\x,{(\x)^2/2/1.0});
\draw [line width=1.pt,color=sqsqsq] (-1.8520469791490313,0.6820975773268491)-- (-0.29818638412888654,1.313141625666884);
\draw [line width=1.pt,color=sqsqsq] (-0.29818638412888654,1.313141625666884)-- (-0.5900051758752041,-0.46712550689850907);
\draw [line width=1.pt,color=sqsqsq] (-0.5900051758752041,-0.46712550689850907)-- (-1.8520469791490313,0.6820975773268491);
\draw [color=sqsqsq,samples=100,rotate around={-90.:(-0.4,0.)},xshift=-0.4cm,yshift=0.cm,line width=1.pt,domain=-4.800000000000001:4.800000000000001)] plot (\x,{(\x)^2/2/0.8});
\draw [line width=1.pt,color=sqsqsq] (-1.8520469791490313,0.6820975773268491)-- (-0.20723282177336982,1.241561154268579);
\draw [line width=1.pt,color=sqsqsq] (-0.20723282177336982,1.241561154268579)-- (-0.4809587382307402,-0.428328743788496);
\draw [line width=1.pt,color=sqsqsq] (-0.4809587382307402,-0.428328743788496)-- (-1.8520469791490313,0.6820975773268491);
\draw [color=sqsqsq,samples=100,rotate around={-90.:(-0.3,0.)},xshift=-0.3cm,yshift=0.cm,line width=1.pt,domain=-4.8:4.8)] plot (\x,{(\x)^2/2/0.6});
\draw [line width=1.pt,color=sqsqsq] (-1.8520469791490313,0.6820975773268491)-- (-0.11905178333858005,1.153066648031105);
\draw [line width=1.pt,color=sqsqsq] (-0.11905178333858005,1.153066648031105)-- (-0.3691397766655495,-0.3726179458393139);
\draw [line width=1.pt,color=sqsqsq] (-0.3691397766655495,-0.3726179458393139)-- (-1.8520469791490313,0.6820975773268491);
\draw [color=sqsqsq,samples=100,rotate around={-90.:(-0.2,0.)},xshift=-0.2cm,yshift=0.cm,line width=1.pt,domain=-4.0:4.0)] plot (\x,{(\x)^2/2/0.4});
\draw [line width=1.pt,color=sqsqsq] (-1.8520469791490313,0.6820975773268491)-- (-0.034536904946319974,1.0422063983561274);
\draw [line width=1.pt,color=sqsqsq] (-0.034536904946319974,1.0422063983561274)-- (-0.2536546550578287,-0.29454140445262844);
\draw [line width=1.pt,color=sqsqsq] (-0.2536546550578287,-0.29454140445262844)-- (-1.8520469791490313,0.6820975773268491);
\begin{scriptsize}
\draw [fill=sqsqsq] (0.,0.) circle (1.0pt);
\draw[color=sqsqsq] (0.08404860843279192,-0.07732013939858942) node {$F$};
\draw [fill=sqsqsq] (-0.8701192695764831,0.4928412084147262) circle (1.0pt);
\draw[color=sqsqsq] (-0.8910416447482203,0.66) node {$E$};
\draw[color=sqsqsq] (-1.3917636666519835,1.5390456856942558) node {$\mathcal{D}$};
\draw[color=sqsqsq] (1.5774300772685765,1.8992141575899442) node {$\mathcal{F}$};
\draw [fill=sqsqsq] (-1.8520469791490313,0.6820975773268491) circle (1.0pt);
\draw [fill=sqsqsq] (-0.29818638412888654,1.313141625666884) circle (1.0pt);
\draw [fill=sqsqsq] (-0.5900051758752041,-0.46712550689850907) circle (1.0pt);
\draw [fill=sqsqsq] (-0.20723282177336982,1.241561154268579) circle (1.0pt);
\draw [fill=sqsqsq] (-0.4809587382307402,-0.428328743788496) circle (1.0pt);
\draw [fill=sqsqsq] (-0.11905178333858005,1.153066648031105) circle (1.0pt);
\draw [fill=sqsqsq] (-0.3691397766655495,-0.3726179458393139) circle (1.0pt);
\draw [fill=sqsqsq] (-0.034536904946319974,1.0422063983561274) circle (1.0pt);
\draw [fill=sqsqsq] (-0.2536546550578287,-0.29454140445262844) circle (1.0pt);
\draw [color=sqsqsq,samples=100,rotate around={-90.:(-0.1,0.)},xshift=-0.1cm,yshift=0.cm,line width=1.pt,domain=-2.4000000000000004:2.4000000000000004)] plot (\x,{(\x)^2/2/0.2});
\draw [color=sqsqsq,line width=1.pt] (-1.8520469791490313,0.6820975773268491)-- (0.04440793370994796,0.8973656136726635);
\draw [color=sqsqsq,line width=1.pt] (0.04440793370994796,0.8973656136726635)-- (-0.13259949371411597,-0.18248432805745662);
\draw [color=sqsqsq,line width=1.pt] (-0.13259949371411597,-0.18248432805745662)-- (-1.8520469791490313,0.6820975773268491);
\draw [fill=sqsqsq] (0.04440793370994796,0.8973656136726635) circle (1.0pt);
\draw [fill=sqsqsq] (-0.13259949371411597,-0.18248432805745662) circle (1.0pt);
\end{scriptsize}
\end{tikzpicture}
\caption{$n=3$.}
    \label{fig.16(A)}
  \end{subfigure}
  \hfill
  \begin{subfigure}[b]{0.45\textwidth}
    \centering
\begin{tikzpicture}[scale=1.5]
\clip(-1.5,-1.5) rectangle (1.5,1.5);
\draw [color=sqsqsq,line width=1.pt] (0.,0.) circle (1.cm);
\draw[color=sqsqsq,line width=1.pt] (2.3693590290862097,1.755629281210456) -- (3.6213785675265227,1.755629281210456);
\draw [color=sqsqsq,samples=100,rotate around={-90.:(-0.5,0.)},xshift=-0.5cm,yshift=0.cm,line width=1.pt,domain=-6.0:6.0)] plot (\x,{(\x)^2/2/1.0});
\draw [color=sqsqsq,line width=1.pt] (-0.7725400139591265,0.6349660832139247)-- (-0.22745998604087334,0.973787427907285);
\draw [color=sqsqsq,line width=1.pt] (-0.22745998604087334,0.973787427907285)-- (-0.7725400139591266,-0.6349660832139247);
\draw [color=sqsqsq,line width=1.pt] (-0.7725400139591266,-0.6349660832139247)-- (-0.22745998604087353,-0.9737874279072848);
\draw [color=sqsqsq,line width=1.pt] (-0.22745998604087353,-0.9737874279072848)-- (-0.7725400139591265,0.6349660832139247);
\draw [color=sqsqsq,samples=100,rotate around={-90.:(-0.4,0.)},xshift=-0.4cm,yshift=0.cm,line width=1.pt,domain=-4.800000000000001:4.800000000000001)] plot (\x,{(\x)^2/2/0.8});
\draw [color=sqsqsq,line width=1.pt] (-0.7725400139591265,0.6349660832139247)-- (-0.02745998604087361,0.9996229034824258);
\draw [color=sqsqsq,line width=1.pt] (-0.02745998604087361,0.9996229034824258)-- (-0.7725400139591263,-0.6349660832139249);
\draw [color=sqsqsq,line width=1.pt] (-0.7725400139591263,-0.6349660832139249)-- (-0.02745998604087363,-0.9996229034824258);
\draw [color=sqsqsq,line width=1.pt] (-0.02745998604087363,-0.9996229034824258)-- (-0.7725400139591265,0.6349660832139247);
\draw [color=sqsqsq,samples=100,rotate around={-90.:(-0.3,0.)},xshift=-0.3cm,yshift=0.cm,line width=1.pt,domain=-4.8:4.8)] plot (\x,{(\x)^2/2/0.6});
\draw [color=sqsqsq,line width=1.pt] (-0.7725400139591265,0.6349660832139247)-- (0.17254001395912635,0.9850025094297905);
\draw [color=sqsqsq,line width=1.pt] (0.17254001395912635,0.9850025094297905)-- (-0.7725400139591265,-0.6349660832139247);
\draw [color=sqsqsq,line width=1.pt] (-0.7725400139591265,-0.6349660832139247)-- (0.1725400139591266,-0.9850025094297905);
\draw [color=sqsqsq,line width=1.pt] (0.1725400139591266,-0.9850025094297905)-- (-0.7725400139591265,0.6349660832139247);
\draw [color=sqsqsq,samples=100,rotate around={-90.:(-0.2,0.)},xshift=-0.2cm,yshift=0.cm,line width=1.pt,domain=-3.2:3.2)] plot (\x,{(\x)^2/2/0.4});
\draw [color=sqsqsq,line width=1.pt] (-0.7725400139591265,0.6349660832139247)-- (0.3725400139591263,0.9280161302473864);
\draw [color=sqsqsq,line width=1.pt] (0.3725400139591263,0.9280161302473864)-- (-0.7725400139591263,-0.6349660832139249);
\draw [color=sqsqsq,line width=1.pt] (-0.7725400139591263,-0.6349660832139249)-- (0.3725400139591265,-0.9280161302473864);
\draw [color=sqsqsq,line width=1.pt] (0.3725400139591265,-0.9280161302473864)-- (-0.7725400139591265,0.6349660832139247);
\draw [color=sqsqsq,samples=100,rotate around={-90.:(-0.1,0.)},xshift=-0.1cm,yshift=0.cm,line width=1.pt,domain=-2.4000000000000004:2.4000000000000004)] plot (\x,{(\x)^2/2/0.2});
\draw [color=sqsqsq,line width=1.pt] (-0.7725400139591265,0.6349660832139247)-- (0.5725400139591265,0.8198767788001338);
\draw [color=sqsqsq,line width=1.pt] (0.5725400139591265,0.8198767788001338)-- (-0.7725400139591264,-0.6349660832139248);
\draw [color=sqsqsq,line width=1.pt] (-0.7725400139591264,-0.6349660832139248)-- (0.5725400139591266,-0.8198767788001337);
\draw [color=sqsqsq,line width=1.pt] (0.5725400139591266,-0.8198767788001337)-- (-0.7725400139591265,0.6349660832139247);
\begin{scriptsize}
\draw[color=sqsqsq] (-0.5603666908641223,1.0075476069923697) node {$\mathcal{D}$};
\draw [fill=sqsqsq] (0.,0.) circle (1.0pt);
\draw[color=sqsqsq] (0.1,0.1060935393153451) node {$F$};
\draw [fill=sqsqsq] (-0.7725400139591265,0.6349660832139247) circle (1.0pt);
\draw [fill=sqsqsq] (3.1205707521503974,1.755629281210456) circle (1.0pt);
\draw[color=sqsqsq] (1.4,1.3) node {$\mathcal{F}$};
\draw [fill=sqsqsq] (-0.22745998604087334,0.973787427907285) circle (1.0pt);
\draw [fill=sqsqsq] (-0.7725400139591266,-0.6349660832139247) circle (1.0pt);
\draw [fill=sqsqsq] (-0.22745998604087353,-0.9737874279072848) circle (1.0pt);
\draw [fill=sqsqsq] (-0.7725400139591265,0.6349660832139247) circle (1.0pt);
\draw [fill=sqsqsq] (-0.02745998604087361,0.9996229034824258) circle (1.0pt);
\draw [fill=sqsqsq] (-0.7725400139591263,-0.6349660832139249) circle (1.0pt);
\draw [fill=sqsqsq] (-0.02745998604087363,-0.9996229034824258) circle (1.0pt);
\draw [fill=sqsqsq] (-0.7725400139591265,0.6349660832139247) circle (1.0pt);
\draw [fill=sqsqsq] (0.17254001395912635,0.9850025094297905) circle (1.0pt);
\draw [fill=sqsqsq] (-0.7725400139591265,-0.6349660832139247) circle (1.0pt);
\draw [fill=sqsqsq] (0.1725400139591266,-0.9850025094297905) circle (1.0pt);
\draw [fill=sqsqsq] (-0.7725400139591265,0.6349660832139247) circle (1.0pt);
\draw [fill=sqsqsq] (0.3725400139591263,0.9280161302473864) circle (1.0pt);
\draw [fill=sqsqsq] (-0.7725400139591263,-0.6349660832139249) circle (1.0pt);
\draw [fill=sqsqsq] (0.3725400139591265,-0.9280161302473864) circle (1.0pt);
\draw [fill=sqsqsq] (-0.7725400139591265,0.6349660832139247) circle (1.0pt);
\draw [fill=sqsqsq] (0.5725400139591265,0.8198767788001338) circle (1.0pt);
\draw [fill=sqsqsq] (-0.7725400139591264,-0.6349660832139248) circle (1.0pt);
\draw [fill=sqsqsq] (0.5725400139591266,-0.8198767788001337) circle (1.0pt);
\end{scriptsize}
\end{tikzpicture}
    \caption{$n=4$.}
    \label{fig.16(B)}
  \end{subfigure}
    \caption{Illustrations of $n$-isoperiodicity.}
  \label{fig.18}
\end{figure}

See Figure \ref{fig.18}. These results, we proved here using pure geometric considerations, thus, are logically independent from \cite{Dragovic-Murad2025a}. Moreover, the case (c) is novel with respect to \cite{Dragovic-Murad2025a}, since $\mathcal{G}$ is not a confocal family of parabolas.
\\

\thanks{\textbf{Acknowledgments}. The authors acknowledge the use of computer algebra systems \textit{Mathematica} and \textit{Geogebra} for various symbolic and numerical computations, and generating the figures in this research.

Authors also express their sincere gratitude to the referees for their constructive comments and useful suggestions.

This research has been partially supported by  the Simons Foundation grant no. 854861, the Serbian Ministry of Science, Technological Development and Innovation and the Science Fund of Serbia grant IntegraRS.}

\end{document}